\documentclass[final,3p,times]{elsarticle}
\usepackage[english]{babel}
\usepackage[utf8]{inputenc}
\usepackage[T1]{fontenc}
\usepackage{helvet}
\usepackage{etoolbox}
\usepackage{graphicx}
\usepackage{url}
\usepackage{titlesec}
\usepackage{amssymb}
\usepackage{amsthm}
\usepackage{mathtools}
\usepackage{subcaption} 
\usepackage{tabularx}
\usepackage{colortbl}
\usepackage{hhline}
\usepackage{float}
\usepackage{times}
\usepackage{listings}
\usepackage{setspace}
\usepackage{enumerate}
\usepackage{lipsum} 
\usepackage{ulem}
\usepackage{textcomp}
\usepackage{stfloats}
\usepackage{verbatim}
\usepackage{booktabs}
\usepackage{mathrsfs}
\usepackage{multirow}
\usepackage{multicol}
\usepackage{subcaption}
\usepackage{adjustbox}
\usepackage{algorithm}
\usepackage{algpseudocode}

\newtheorem{thm}{\bf Definition}[section]
\newtheorem{remark}{\bf Remark}[section]
\newtheorem{pro}{\bf Proposition}[section]
\newtheorem{theorem}{\bf Theorem}[section]

\newcommand{\vecx}{\mathbf{x}}

\newcommand{\vecX}{\mathbf{X}}

\newcommand{\vecN}{\mathbf{N}}

\newcommand{\vecz}{\mathbf{z}}
\newcommand{\vecy}{\mathbf{y}}

\newcommand{\vecc}{\mathbf{c}}

\newcommand{\vectheta}{\boldsymbol\theta}

\newcommand{\argmax}{\arg\max}
\newtheorem{assumption}{Assumption}
\newtheorem{lemma}{Lemma}

\journal{}

\begin{document}

\begin{frontmatter}



\title{ $\beta$-integrated local depth and corresponding partitioned local depth representation}


\author[1]{Siyi~Wang}
\author[2]{Alexandre~Leblanc}
\author[1]{Paul~D.~McNicholas}
\affiliation[1]{organization={McMaster University},
            addressline={1280 main St west}, 
            city={Hamilton},
            postcode={L8S 4L8}, 
            state={Ontario},
            country={Canada}}
\affiliation[2]{
  organization={University of Manitoba},
  addressline={66 Chancellors Circle},
  city={Winnipeg},
  postcode={R3T 2N2},
  state={Manitoba},
  country={Canada}
}
\begin{abstract}
A novel local depth definition, $\beta$-integrated local depth ($\beta$-ILD), is proposed as a generalization of the local depth introduced by Paindaveine and Van Bever \cite{paindaveine2013depth}, designed to quantify the local centrality of data points. $\beta$-ILD inherits desirable properties from global data depth and remains robust across varying locality levels. A partitioning approach for $\beta$-ILD is introduced, leading to the construction of a matrix that quantifies the contribution of one point to another's local depth, providing a new interpretable measure of local centrality. These concepts are applied to classification and outlier detection tasks, demonstrating significant improvements in the performance of depth-based algorithms.
\end{abstract}


\end{frontmatter}



\section{Introduction}
For multivatiate datasets, measuring the centrality of a subset of data or defining a central region within that data is a problem at the center of numerous statistical tasks. A widely-used approach to address this is through the concept of data depth. In general, a depth function is a bounded measure that provides a meaningful center-outward ranking of points in a multidimensional space. This is formally summarized in Definition \ref{def:DATADEPTH}. Some notable examples of depth functions include halfspace depth (Tukey, $1975$ \cite{Tukey1975}), simplicial depth (Liu, $1990$ \cite{liu1990notion}), spatial depth (Serfling, $2002$ \cite{serfling2002depth}), projection depth (Zuo and Serfling, $2000$ \cite{zuo2000general}) and Mahalanobis depth.
\begin{thm}[Statistical depth function]\label{def:DATADEPTH}
    Let $\mathcal{P}$ denote the class of probability distributions on $\mathbb{R}^d$, and let $P$ be any distribution from $\mathcal{P}$. A depth function $D(\bullet\mid P):\mathbb{R}^d\to[0,\infty)$ is a bounded, non-negative function which satisfies the following properties \cite{zuo2000general}:
    \begin{enumerate}[(P1)]
        \item \textbf{$\mathcal{T}$-invariance}: Define a class of transformations $ \mathcal{T} $ over $\mathbb{R}^d$, for any transformation $T\in\mathcal{T}$, there is $D(T(\vecx) \mid P_{T(\vecX)}) = D(\vecx \mid P_{\vecX})$ where $P_\vecX$ is the distribution of a random variable $\vecX$.
        \item \textbf{Maximality at center}: If the distribution $P$ is centrally symmetric about a point $\vectheta\in \mathbb{R}^d$, then $\vectheta$ attains the maximum depth value, i.e., $D(\vectheta\mid P)=\sup_{x\in \mathbb{R}^d}D(\vecx\mid P)$. $\vectheta$ is referred to as the depth median.
        \item  \textbf{Monotonicity on rays}: For any $P$ with a depth median $\vectheta$, $D(\vecx\mid P)\leq D((1-\gamma)\vectheta+\gamma\vecx\mid P)$ for any $\vecx\in\mathbb{R}^d$ and $\gamma\in [0,1]$.
        \item  \textbf{Vanishing at infinity}: For any $P$, $D(\vecx\mid P)\to 0$ as $||\vecx||\to \infty$.
    \end{enumerate}
\end{thm}
\begin{remark}
The property P1 as originally introduced by Zuo and Serfling is affine invariance \cite{zuo2000general}, but some depth definitions relax this condition to invariance under similarity transformations, such as location, scale, rotation, and reflection, as discussed in \cite{serfling2002depth, cuevas2009depth, ramsay2019integrated}. To allow for this different requirement, we adopt a more general version referred to as $\mathcal{T}$-invariance, following the review of Mosler and Mozharovskyi \cite{mosler2022choosing}. \textcolor{black}{In this paper, we allow $\mathcal{T}$ to represent any transformation family that relaxes affine invariance (e.g., similarity transformations).}
\end{remark}

Given for above properties, data depth can be seen as a global concept. Any depth function generally defines a universal center and depth values decrease monotonically in all directions from that center. As a result, traditional depth functions will ignore local characteristics of a distribution, making them more appropriate for distributions with unimodal and convex support. In contrast, many real-world distributions exhibit non-convex support and/or multimodality, and tasks such as clustering often require attention to the local features of the data. To extend the applicability of data depth to such scenarios, researchers introduced the concept of local depth, which assesses depth values at a local level, capturing more localized structural information.

One of the earliest examples of local depth in the literature is the kernelized spatial depth (KSD), introduced in $2009$ \cite{chen2008outlier} and applied to outlier detection. KSD calculates spatial depth after kernelizing the distances between points using kernels such as the Gaussian kernel, $$\kappa(\mathbf{x},\mathbf{y})=(\sqrt{2\pi}\sigma)^{-d}\frac{\exp(-\|\mathbf{x}-\mathbf{y}\|^{2})}{2\sigma^{2}},$$ for $d$-dimensional data. The bandwidth parameter $\sigma$ determines the level of localization. As $\sigma$ approaches zero, the KSD converges to the underlying data density \cite{dutta2016multi}.

The term ``local depth'' was proposed by Agostinelli and Romanazzi \cite{agostinelli2011local} in $2011$, as they introduced localized versions of both halfspace depth and simplicial depth. In their approach, they introduced a parameter, $\tau$, which controls the level of localization. Although not directly relying on distance to localize their depth, a smaller value of $\tau$ results in smaller data regions being considered, leading to a more localized depth measure.

In $2013$, Paindaveine and Van Bever \cite{paindaveine2013depth} introduced a general framework for defining local depth that is compatible with most depth functions. When calculating the local depth of any point $\vecx \in \mathbb{R}^d$, the core idea is to create a new distribution that is centrally symmetric about $\vecx$. Specifically, they considered the distribution $0.5P_\vecX + 0.5P_{2\vecx - \vecX}$ for a given distribution $P_\vecX$. They then introduced a locality parameter $\beta \in (0, 1]$, and the local depth of $\vecx$ is computed with respect to the distribution conditioned on the smallest central region of the symmetric distribution that has a probability mass of at least $\beta$, $P_{\vecx}^\beta$. For simplicity, we will refer to this as the $\beta$-local depth of the point $\vecx$ throughout this paper.

Paindaveine and Van Bever's methodology extends to data depths of the integrated type. For example, Piana and Svarc \cite{Piana2022} introduced a local version of the integrated dual depth (IDD) \cite{cuevas2009depth}. The IDD combines one-dimensional projections with one-dimensional depth measures. For a random element $\vecx$ in a Banach space with distribution $P$, IDD is defined as:
\begin{equation}
    IDD(\vecx| P) = \int D_{1d}(f(\vecx)\mid P_f) dQ(f),
\end{equation}
where $D_{1d}$ is a univariate simplicial depth function, $P_f$ is the distribution of $f(\vecX)$, and $Q(f)$ is a probability measure over the dual space. The local version of IDD constructs the neighborhood of an observation within each one-dimensional projection space separately, rather than in the original data space, offering enhanced computational efficiency compared to the $\beta$-local simplicial depth.

As observed, the local depth definitions previously discussed require a tuning parameter to control their locality, which can be challenging to optimize. In contrast, Berenhaut et al. \cite{berenhaut2022social} proposed a unique non-parametric approach known as local community depth (LCD). \textcolor{black}{Rather than restricting the computation of depth values to a pre-defined local region}, this method captures local geometric centrality by analyzing the local focus region constructed for every pair of points. This pairing approach, which relies on ordinal information (i.e., the ordering of dissimilarities rather than their exact values), is similar to the ideas behind lens depth~\cite{liu2011lens, kleindessner2017lens} and its generalized form, $\beta$-skeleton depth~\cite{mengta2018}. Lens depth is based on the intersection of two balls centered at each of the two points, whereas the local focus region in LCD is defined as their union. Although LCD measures local centrality, it also considers the relative positions of all points in the dataset. Without a locality control parameter, the LCD value of a point can be influenced by its global position. For instance, in Figure \ref{fig:LCDsmi}, points in the eyes that are closer to the nose have higher LCD values compared to the points that are more central within the eyes.

\begin{figure}[ht]
    \centering
    \begin{subfigure}[b]{0.48\textwidth}
        \centering
        \includegraphics[width=\textwidth]{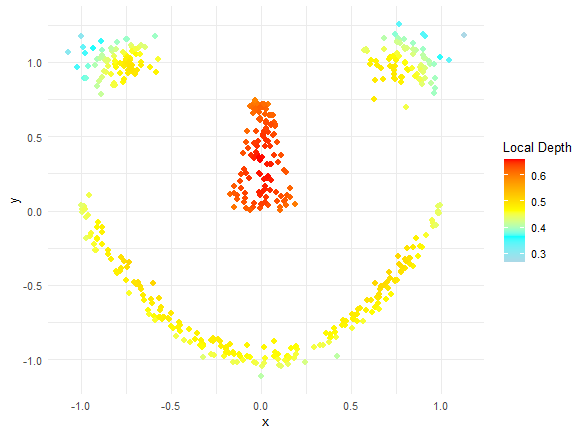}
        \caption{}
        \label{fig:LCDsmi}
    \end{subfigure}
    \hfill
    \begin{subfigure}[b]{0.48\textwidth}
        \centering
        \includegraphics[width=\textwidth]{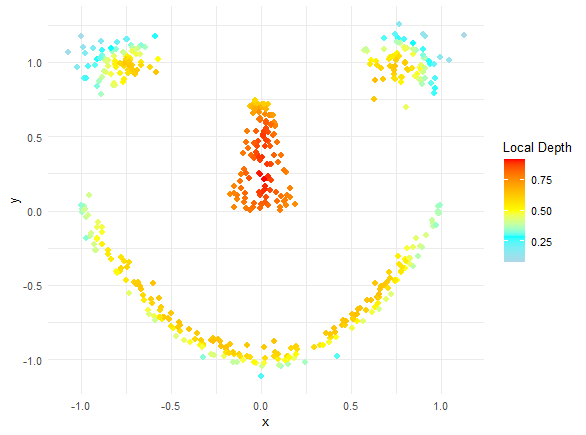}
    \caption{}
    \label{fig:ILDsmi}
    \end{subfigure}
\caption{Examples of LCD and $\beta$-ILD values in the simulated smiley dataset: (a) LCD (b) $\beta$-ILD.}
\end{figure}
The development of local depth motivates us to propose an extension of previous local depth definitions called $\beta$-integrated local depth ($\beta$-ILD). This new contribution builds on $\beta$-LD and is designed to be more robust to changes in the locality parameter. Moreover, $\beta$-ILD can also resemble the performance of LCD by considering all points in the dataset rather than being confined to predefined neighborhoods, one example is shown in Figure \ref{fig:ILDsmi}. This flexibility allows $\beta$-ILD to inherit desirable properties from both $\beta$-LD and LCD. 
The main contributions of this paper are summarized into $3$ parts:
\begin{enumerate}[1] 
\item The concept of $\beta$-integrated local depth ($\beta$-ILD) is proposed, treating the locality parameter $\beta$ as an independent variable and integrating the depth function with respect to it. Theoretical properties inherited from global depth and $\beta$-LD are analyzed, including consistency, continuity as a function of $\vecx$, $\mathcal{T}$-invariance, centrality at extreme locality levels, and vanishing at infinity. 
\item A connection between $\beta$-ILD and the partitioned local depth (PaLD) framework introduced in LCD is established. The column sums of the PaLD matrix constructed based on $\beta$-ILD (referred to as PILD) are shown to serve as an interpretable and effective measure of local centrality. 
\item The proposed concepts are applied to classification and outlier detection tasks. A new depth-based classification method and a depth-based similarity matrix are introduced based on the constructed PILD matrix, demonstrating impressive improvements on existing depth-based methods. 
\end{enumerate}
The remainder of this paper is organized as follows: Section \ref{Sec: betaLD} reviews the core concepts of $\beta$-LD and introduces $\beta$-ILD, accompanied by examples demonstrating the robustness of $\beta$-ILD to changes in the locality parameter. Section \ref{sec:property} discusses the theoretical properties of $\beta$-ILD and presents its PaLD-based representation. Section \ref{sec:application} provides application examples in classification and outlier detection. Finally, Section \ref{sec:conclu} concludes the paper and outlines directions for future work. Proofs are included in the Appendix.

\section{From $\beta$-local depth to $\beta$-integrated local depth}\label{Sec: betaLD}
In this section, we first review the core concepts of $\beta$-local depth and then extend it to the $\beta$-integrated local depth. For simplicity, we assume that the depth function under discussion is bounded between $0$ and $1$.
\begin{thm}[Central region]
    A central region (or level set) at level $\alpha\in (0,1]$ of any depth function $D$ with respect to a distribution $P$ is $R_\alpha(P)=\{\vecx \in \mathbb{R}^d: D(\vecx \mid P)\geq \alpha\}$. It can also be parameterized from a probability perspective, which is commonly used in the context of local depth. In this view, the $\beta$-central region is defined as the smallest region enclosed by the contour of depth $\alpha$ which achieves a probability mass of at least $\beta\in (0,1]$, formally represented as:
    $$
   R^\beta(P)=\bigcap\limits_{\alpha\in A(\beta)}R_\alpha(P),\quad\text{where}\quad A(\beta)=\{\alpha:P(R_\alpha)\geq\beta\} .
    $$
\end{thm}
For simplicity in writing, we will use $ R^\beta $ as an abbreviation for $ R^\beta(P) $ for the remainder of the paper.
\begin{thm}[$\beta$-local depth]
   For a distribution $P_\vecX$, the $\beta$-local depth of any location $\vecx \in \mathbb{R}^d$ is defined based on a mixture distribution $0.5P_\vecX + 0.5P_{2\vecx - \vecX}$. The corresponding $\beta$-central region of this mixture distribution is denoted as $R^\beta_\vecx$. Let $D$ be a given depth function, then the $\beta$-local depth of $\vecx$ is defined as: 
   \begin{equation}
       LD^\beta(\vecx \mid P_\vecX)=D(\vecx \mid P_\vecx^\beta),
   \end{equation}
where $P_\vecx^\beta$ is the probability distribution of $P_\vecX$ conditioned on the region $R^\beta_\vecx$. 
\end{thm}
In practice, the sample version of $\beta$-LD is required. Consider a point $\vecz$ and a dataset $\mathbf{X}$ with $n$ observations. To compute its sample $\beta$-LD, we construct the set  
$$\mathbf{X}_{Rz} = \mathbf{X} \cup \{ 2\vecz - \vecx_j \, | j = 1, 2, \ldots, n\},$$ 
and note that $\vecz$ is the depth median of $\mathbf{X}_{Rz}$ due to its central symmetry.

\begin{thm}[$\beta$-neighborhood and sample $\beta$-local depth] \label{def:betanbr}
The $\beta$-neighborhood of $\vecz$, denoted as $\vecN^\beta_\vecz$, is defined as the set of $\lceil n\beta \rceil$ points $\vecx_i \in \mathbf{X}$ that have the highest depth values relative to $\mathbf{X}_{Rz}$, i.e., the points with the largest $D\left(\vecx_i \mid \mathbf{X}_{Rz}\right)$. Here, $\lceil \cdot \rceil$ denotes the ceiling function.
Formally, let $P^{n}_\vecX$ represent the empirical distribution of the points in $\mathbf{X}$, and $P^{\beta n}_{\vecz}$ represent the empirical distribution of points in the $\beta$-neighborhood of the point $\vecz$. The sample $\beta$-LD function can be written as 
\begin{equation}
  LD^\beta(\vecz \mid P^n_\vecX) =D(\vecz \mid P^{\beta n}_{\vecz}).
\end{equation}

\end{thm}
The sample $\beta$-LD can be sensitive to small changes in $\beta$, particularly when $\beta$ is small. This sensitivity means that even slight variations in the selection of $\beta$ can lead to significant changes in the local depth values for a given point, potentially affecting the results of applications that utilize this concept. To illustrate this sensitivity, we provide a toy example in Figure \ref{fig:LDlimit}, where the local depth (LD) value of the target point fluctuates dramatically as $\beta$ increases from $0.1$ to $0.3$ \footnote{A Shiny app is provided: \url{https://lytgysrn.shinyapps.io/LDvsILD/}}.

In this example, the depth function used is the projection depth (PD), defined as
\begin{eqnarray}\label{eq:pd}
PD(\vecx\mid\mathbf{X})=\left(1+\sup_{\Vert \boldsymbol{u} \Vert=1}\frac{|\boldsymbol{u}^T \vecx-\text{Med}(\boldsymbol{u}^T\mathbf{X})|}{\text{MAD}(\boldsymbol{u}^T\mathbf{X})}\right)^{-1},
\end{eqnarray}
where $\mathrm{Med}$ stands for the median and $\mathrm{MAD}$ is the median absolute deviation, and 
$\boldsymbol{u}^{T}\mathbf{X}=\{\boldsymbol{u}^{T}\vecx_{1},\ldots,\boldsymbol{u}^{T}\vecx_{n}\}.$
It is computed using Liu and Zuo's exact algorithm for 2D data \cite{liu2013exactly,liu2014computing}. We chose PD because it satisfies the assumptions made in this section, i.e., its values are bounded within $[0, 1]$ and it has all the properties outlined in Definition \ref{def:DATADEPTH}. These characteristics help avoid any potential ambiguities in our discussion.
Moreover, PD provides an efficient approach for defining the $\beta$-neighborhood. Specifically, it can be observed that
 \begin{eqnarray}
     PD(\vecx_q\mid \vecX_{Rz})=\left(1+\sup_{\Vert \boldsymbol{u} \Vert=1}\frac{|\boldsymbol{u}^T \vecx_q-\boldsymbol{u}^T \vecz|}{\text{Med}(|\boldsymbol{u}^T\mathbf{X}-\boldsymbol{u}^T \vecz|)}\right)^{-1}.
 \end{eqnarray}
This formulation shows that PD can define the $\beta$-neighborhood for any point without the need to explicitly construct the reflection data set.

\begin{figure}
    \centering
    \includegraphics[width=1\linewidth]{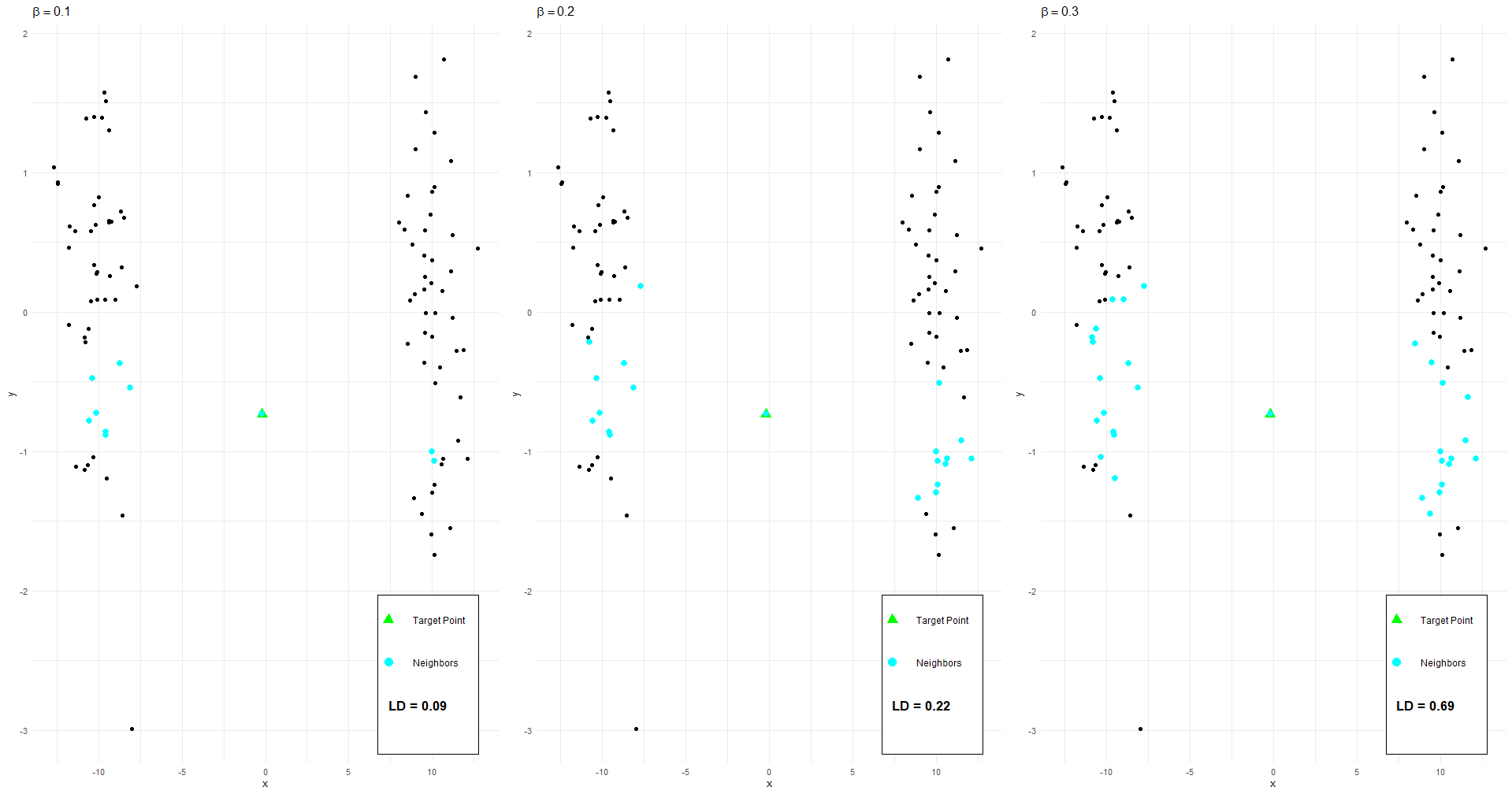}
    \caption{Local projection depth for a selected point in a toy example at three different $\beta$ levels. The LD value changes from $0.09$ to $0.22$ and then to $0.69$ as $\beta$ increases from $0.1$ to $0.2$ and $0.3$, respectively.}
    \label{fig:LDlimit}
\end{figure}
\subsection{Locality integrated local depth}
\textcolor{black}{Instead of treating the locality level $\beta$ as a fixed parameter, we reinterpret the local depth as a function of $\beta$ for a fixed point $\vecx$. That is, we write $LD(\beta \mid \vecx, P_\vecX)$ to emphasize the dependency on $\beta$.} Building on this alternative view of the local depth function, we now introduce a weighted average of the $\beta$-local depth of a point $\vecx$ over a range of locality levels, which we will refer to as $\beta$-integrated local depth.
\begin{thm}[$\beta$-integrated local depth]\label{def:BILD}
For a distribution $ P_\vecX $, the $\beta$-integrated local depth of a point $ \vecx $ is defined as 
    \begin{equation}
            ILD(\vecx \mid P_\vecX, W) = \int_0^1 LD(\beta \mid \vecx, P_\vecX) \, dW(\beta),
    \end{equation}
where $W$ is a probability measure such that $ W((0,1]) = 1 $. If $ W $ is absolutely continuous with respect to the Lebesgue measure on $(0, 1]$, then a density $ w(\beta) $ can be defined such that $ dW(\beta) = w(\beta) d\beta $, and it serves as the weighting function across all values of $\beta$ on the unit interval.
\end{thm}
Note that the $\beta$-ILD generalizes the $\beta$-LD function by simultaneously considering a range of $\beta$ values. Specifically, if the weighting measure $W$ is degenerate at a single value $B \in (0,1]$, i.e., $W(\{\beta = B\}) = 1$, then the $\beta$-ILD reduces to the $\beta$-LD at locality level $B$. Additionally, if $W$ is concentrated at $\beta = 1$, $\beta$-ILD coincides with the global depth function.

\textcolor{black}{
Let $\vecX = \{\vecx_1,\vecx_2,\ldots,\vecx_n\}$ be a random sample from distribution $P_\vecX$, and denote by $P_\vecX^n$ the corresponding empirical distribution. The sample $\beta$-ILD of a point $\vecx$ with respect to $P_\vecX^n$ is given by
\begin{equation}\notag
SILD(\vecx\mid \vecX, W)=ILD(\vecx \mid P_\vecX^n, W).
\end{equation}
The sample $\beta$-local depth is a step function in $\beta$. Let $\beta_1 = \frac{n_0}{n}$, where $n_0$ is the minimum number of points needed to compute local depth (e.g., $n_0 = 3$ for spatial depth). We define a grid of $\beta$ values as $\beta_i = \beta_1 + \frac{i - 1}{n}$ for $i = 1, \ldots, b$, with $\beta_b = 1$, so that each increment corresponds to adding one point to the neighborhood. For completeness, we set $\beta_0 = \beta_1 - \frac{1}{n}$ as the lower bound of the first integration interval, although no depth is computed at this level. Under this setup, the sample $\beta$-integrated local depth (SILD) takes the form of summation
\begin{eqnarray}
    SILD(\vecx\mid\vecX, W)=\sum_{i=1}^{b}LD(\beta_{i} \mid \vecx, \vecX) \int_{\beta_{i-1}}^{\beta_{i}} w(\beta) \, d\beta.
\end{eqnarray}
where $\int_{\beta_0}^1 w(\beta)d\beta=1$.}
\begin{remark}
    \textcolor{black}{In the sample version of any depth function (i.e., when depth is defined on a finite set of observed points), a minimum of three sample points is required to obtain a nontrivial central‐outward ordering; with fewer than 3 points, ``centrality'' becomes trivial or undefined.} However, specific depth definitions may require much more points for computation. For example, geometric-based depths such as halfspace and simplicial depth require at least $ d+1 $ and $ d+2 $ points, respectively, when working with d dimensional data. 
\end{remark}

\begin{figure}[t]
    \centering
    \includegraphics[width=0.7\linewidth]{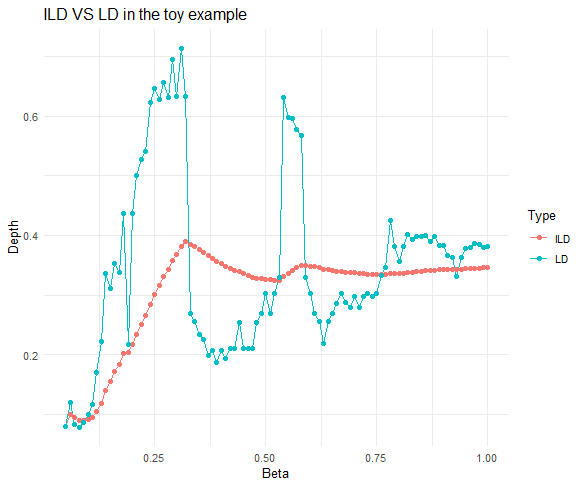}
    \caption{Comparison between $\beta$-ILD and $\beta$-LD across different $\beta$ values (or $B$ in $\beta$-ILD). The depth values are computed for the target point with respect to the data points shown in Figure \ref{fig:LDlimit}. The red line represents the $\beta$-ILD, while the blue line represents the $\beta$-LD.}
    \label{fig:ILDvsLD}
\end{figure}
The $\beta$-ILD adjusts the influence of different locality levels through an appropriate choice of the probability measure $W$ or, more intuitively, the weighting function $ w $. For instance, consider $ w(\beta) $ uniformly distributed over the range $ (B_0, B_1] $
\begin{eqnarray} \label{eq:weight}
    w_B(\beta) = 
\begin{cases} 
\frac{1}{B_1-B_0} & \text{if } B_0 < \beta \leq B_1, \\
0 & \text{otherwise.}
\end{cases}
\end{eqnarray}
In this context, the parameter $B_1$ in $\beta$-ILD serves a role analogous to the parameter $\beta$ in $\beta$-LD, controlling the extent of locality, and unless otherwise specified, $B_0$ is set to the minimal locality level $\beta_0$ as defined above. For a point to exhibit a high $\beta$-ILD value under the weighting function $ w_B $, it should consistently show high local depth across the locality levels $B_0$ up to $ B_1$. This approach ensures that a point is considered locally deep at a certain locality level if it maintains relative centrality across many smaller neighborhoods within that level.
To avoid any potential ambiguity, we refer to the uniformly-weighted $\beta$-ILD when the weighting function $w_B(\beta)$ is used.

The comparison between $\beta$-LD and $\beta$-ILD with the weighting function $w_B(\beta)$ is illustrated in Figure \ref{fig:ILDvsLD}. As shown, the $\beta$-LD curve exhibits substantial fluctuations at certain $\beta$ values, particularly around $\beta = 0.3$ and $\beta = 0.6$. In contrast, the $\beta$-ILD, computed as a weighted integral over all $\beta$-levels, smooths out these variations and yields a more stable assessment of centrality. 

We fix a point $\vecx$ and data set $\vecX$ for the following proposition, which summarizes the smoothing behavior of the integrated local depth.  This allows us to write the sample $\beta$-local depth and $\beta$-integrated local depth (with weighting function $w_B$) simply as $LD(\beta)$ and $ILD(\beta)$. Here, $ILD(\beta)$ means the upper limit $B_1$ in~\eqref{eq:weight} is set to $\beta$. For simplicity, we omit the minimum sample size requirement for depth computation, and let $\beta_i = i/n$ for $i = 1, \ldots, b$, with $B = \beta_b$ as the upper bound of the considered locality range. Under this setup, we have
\begin{pro}\label{pro:1}
For any given sample depth function $ D $, the corresponding $\beta$-integrated local depth serves as a smoothing of the sample $\beta$-local depth in terms of its variability across locality levels. Specifically, define
$$
\Delta = \max_{1 < i \le b} \bigl| LD(\beta_i) - LD(\beta_{i-1}) \bigr| ,\quad\Delta^\dagger
=\max_{1\le i<j\le b}\bigl|LD(\beta_i)-LD(\beta_j)\bigr|
$$
as the maximum jump between successive $\beta$-LD values up to level $B$, and the maximum gap in local depth between any two levels, respectively. Then the corresponding $\beta$-ILD values satisfy
\begin{eqnarray}
    \max_{1 < i \le b} \bigl| ILD(\beta_i) - ILD(\beta_{i-1}) \bigr| &\;\le\;& \tfrac{1}{2} \Delta,\\
   \bigl|ILD(\beta_i)-ILD(\beta_{i-1})\bigr| &\;\le\;& \frac{\Delta^\dagger}{i}.
\end{eqnarray}
\end{pro}
That is, the integration significantly reduces the maximum adjacent fluctuation, and although the fluctuation need not drop at every level, it is bounded by a term that tends to zero as the number of locality levels involved increases.
\section{Properties}\label{sec:property}
\newtheorem{property}{Property}
In this section, the key properties of $\beta$-ILD are presented. We begin with Theorem \ref{thm:consis1}, derived from the work of Paindaveine and Van Bever \cite{paindaveine2013depth}, which establishes the consistency of $\beta$-LD for the distribution and its continuity with respect to $\beta$. These properties are foundational for extending the theoretical results of $\beta$-LD to $\beta$-ILD.
Throughout this section, we assume that the probability measure $W$ is absolutely continuous with respect to the Lebesgue measure on $(0,1]$, allowing it to be represented by a bounded and measurable weighting function $w$ on $(0,1]$. For simplicity, the $\beta$-ILD is denoted as $ILD(\vecx \mid P)$, with $W$ fixed throughout the discussion to provide a consistent framework.
To facilitate the subsequent discussion, the following assumption is made regarding the depth function $D(\vecx \mid P)$ used for local depth:
\begin{assumption}\label{assump:depth} The depth function $D(\vecx \mid P)$ satisfies the properties listed in Definition \ref{def:DATADEPTH}, where $\mathcal{T}$ represents the family of similarity transformations, and exhibits weak continuity, i.e., $|D(\vecx \mid P^n) - D(\vecx \mid P)| \to 0$ as $P^n$ converges weakly to $P$. \end{assumption}
\begin{theorem}[Consistency of $\beta$-LD \cite{paindaveine2013depth}]\label{thm:consis1}
      Let $\vecX \in \mathbb{R}^d$ be a random vector, and let $P$ be an absolutely continuous distribution. Suppose a depth function $ D(\cdot \mid P)$ satisfies Assumption \ref{assump:depth} . Then, for any sequence $\beta_n \to \beta$, the $\beta$-local depth function satisfies
    \[
    |LD^{\beta_n}(\vecx \mid P^{n}) - LD^{\beta}(\vecx \mid P)| \xrightarrow[n \to \infty]{a.s.} 0.
    \]
\end{theorem}
\begin{remark}
    Lemma A.1 from the same work shows that the mapping $\beta \to P(\mathcal{B} \cap R^\beta_\vecx)$ is continuous for any Borel set $\mathcal{B}$.
\end{remark}
From Theorem \ref{thm:consis1}, the consistency of $\beta$-ILD can be established using the Dominated Convergence Theorem (DCT), a widely employed method in proving the consistency of local or global depth functions involving integrals (see \cite{Piana2022,nagy2016integrated}).

\begin{property}[Consistency]\label{prop:consis}
    Let $\vecX \in \mathbb{R}^d$ be a random vector, and let $P$ be an absolutely continuous distribution. Suppose a depth function $ D(\cdot \mid P) $ satisfies Assumption \ref{assump:depth}. Then the $\beta$-integrated local depth function satisfies
    \[
    |ILD(\vecx \mid P^{n}) - ILD(\vecx \mid P)| \xrightarrow[n \to \infty]{a.s.} 0.
    \]
\end{property}
In addition to distributional consistency, it is equally important to analyze the continuity of the depth function with respect to the point $\vecx$.
\begin{property}[Continuity as a function of $\vecx$]\label{prop:ctn}
  Let $\vecX \in \mathbb{R}^d$ be a random vector, and let $P$ be an absolutely continuous distribution. Suppose a quasi-concave depth function $ D(\cdot \mid P) $ satisfies Assumption \ref{assump:depth}. Then, both $\beta$-local depth function $LD^\beta(\cdot|P)$ and $\beta$-integrated local depth function $ILD(\cdot|P)$ are continuous as a function of $\vecx$ with a fixed probability measure $W$ throughout all positions.
\end{property}
With the consistency and continuity of $\beta$-ILD established for certain depth functions under absolutely continuous distribution $P$, the next step is to examine whether the properties outlined in Definition \ref{def:DATADEPTH} hold for local depth functions. If these properties do not directly apply, we will identify suitable alternatives that align with the characteristics of local depth functions.

For the P1, if a depth function $ D $ is invariant under any transformation $ T \in \mathcal{T} $, the $\beta$-local depth will naturally retain this invariance at any locality level. Specifically, for any transformation $ T \in \mathcal{T} $, we have
\[
D(T(\vecx) \mid P^\beta_{T(\vecx)}) = D(\vecx \mid P^\beta_{\vecx}).
\]
Moreover, integrating over $\beta$ does not affect this invariance, meaning that the $\beta$-integrated local depth also retains this property.
\begin{property}[$\mathcal{T}$-invariance]\label{prop:inva}
    If a depth function $ D $ is invariant under any transformation $ T \in \mathcal{T} $, and the probability measure $W$  remains unchanged under any transformation of $P$ then the $\beta$-integrated local depth is also invariant for all $ T \in \mathcal{T} $. That is,
    \[
    ILD(T(\vecx) \mid P_{T(\vecX)}) = ILD(\vecx \mid P_{\vecX}).
    \]
\end{property}

Local depth functions focus primarily on local centrality, meaning they generally do not define a global depth median, and thus properties P2 and P3 are not necessarily satisfied. Instead, the centrality of $\beta$-LD emerges at extreme locality levels, as emphasized by Paindaveine and Van Bever \cite{paindaveine2013depth}. When the local depth function is continuous with respect to $\beta$, there is
\[
\lim\limits_{B \to 0} \int_0^B LD(\beta) \, dW(\beta) = \lim\limits_{\beta \to 0} LD(\beta),
\]
where $ W(0, B] = 1 $. For simplicity, we denote $ \int_0^B LD(\beta) \, dW(\beta)$ as $ ILD^B(\vecx \mid P) $ in latter expression.

\begin{property}[Centrality in extreme locality level]\label{prop:central}
    Assume a depth function \(D\), satisfying the properties in Definition \ref{def:DATADEPTH}, is uniformly weakly continuous (as defined by \(Q_1^+\) in \cite{paindaveine2013depth}), meaning that for any two sequences of absolutely continuous distributions \((P_n)\) and \((P_n')\) such that
    \[
    |P_n(\mathcal{B}) - P_n'(\mathcal{B})| \to 0 \quad \text{for any Borel set } \mathcal{B},
    \quad \text{implies} \quad
    |D(\vecx \mid P_n) - D(\vecx \mid P_n')| \to 0 \quad \text{for any } \vecx \in \mathbb{R}^d.
    \]
    Suppose further that the distribution \(P\), with probability density function \(f\), is absolutely continuous and centrally symmetric about \(\vectheta\). If the depth function attains its maximum at \(\vectheta\), and this maximum value \(c_D\) is independent of the specific distribution \(P\), then for any \(\vecx \in \mathrm{Supp}_{+}(f) = \{\vecx \mid f(\vecx) > 0\}\), we have
    \[
    \lim\limits_{B \to 0} ILD^B(\vecx \mid P) = \lim\limits_{\beta \to 0} LD^\beta(\vecx \mid P) \to c_D.
    \]
\end{property}
Note that the above conditions are satisfied by many commonly used depth functions, such as the halfspace depth ($c_D = 0.5$) and projection depth ($c_D = 1$).

For the vanishing-at-infinity property, we formalize the fact that if the local depth vanishes at infinity, then so does the integrated local depth. We further summarize several widely used depth functions whose $\beta$-LD satisfies this property.
\begin{property}[Vanishing at infinity] \label{prop:vanish}
If, for any distribution $P$ whose mass does not escape to infinity, the local depth $LD^\beta(\vecx \mid P) \to 0$ as $\|\vecx\| \to \infty$ for all $\beta\in(0, 1]$, then the integrated local depth $ILD^B(\vecx \mid P) \to 0$ as $\|\vecx\| \to \infty$ for any $B\in(0, 1]$.
\end{property}
\begin{pro}\label{pro:2}
For any distribution $P$ whose mass does not escape to infinity, we have $LD^\beta(\vecx \mid P) \to 0$ as $\|\vecx\| \to \infty$ for many widely used depth functions. Following the classification of depth function types by Zuo and Serfling~\cite{zuo2000general}, notable examples include Type~A and Type~D depths based on probability content, such as simplicial depth and halfspace depth, Type~B depths based on the expectation of a measure of dispersion, such as $L^p$ depth, and spatial depth.
\end{pro}
\subsection{View of partitioned local depth}\label{sec: PALD}
Although $\beta$-LD does not satisfy monotonicity, the monotonicity of the underlying depth function $D$ is reflected in the nested structure of central regions for symmetric distributions. Specifically, if the depth function $D$ is monotone along rays from its deepest point, we have $R^\beta_\vecx \subseteq R^{\beta'}_\vecx$ for $\beta' > \beta$. In applications, such as affine K-nearest neighbor classification \cite{paindaveine2015nonparametrically} and depth-based local center clustering \cite{wang2025depth}, $D(\vecz \mid 0.5P_\vecX + 0.5P_{2\vecx - \vecX})$ serves as a non-symmetric similarity measure from $\vecz$ to $\vecx$ under the distribution $P_\vecX$, and is monotonic along rays from $\vecx$.

While this depth-based similarity measure helps define $R^\beta_\vecx$ for local depth calculations, the $\beta$-LD does not exhibit monotonicity on its own. To illustrate how the monotonicity of depth-based similarity affects the local depth value, we introduce Partitioned Integrated Local Depth (PILD), inspired by the concept of Partitioned Local Depth (PaLD) in LCD \cite{berenhaut2022social}. In simple terms, the PaLD of a point $ \vecz $ relative to another point $ \vecx $ represents the contribution of $ \vecz $ to the local depth value of $ \vecx $.
For $\beta$-ILD, we can define the sample PILD as
\begin{eqnarray}\label{eq:Cohesion_ild}
         PILD(\vecz \mid \vecx, \vecX) =\sum_{i=1}^{b}\frac{\mathbf{1}\{\vecz \in \vecN_\vecx^{\beta_{i}}\} LD(\beta_{i} \mid \vecx,\vecX)}{\lceil n\beta_{i} \rceil}\int_{\beta_{i-1}}^{\beta_{i}}w(\beta)d\beta,
\end{eqnarray}
where $ \mathbf{1}\{\vecz \in \vecN_\vecx^{\beta_{i}}\} $ indicates whether $\vecz$ is in the $\beta_i$-neighborhood of $\vecx$, i.e., $\vecN_\vecx^{\beta_i}$. By default, we take $\beta_b = 1$; if the locality level is constrained within $B$, we set $\beta_b = B$ and denote the variant by $PILD^B$. The denominator $\lceil n\beta_{i} \rceil$ indicates that, for a fixed locality level $\beta_i$, the $\beta$-LD value at $\vecx$ is treated as being equally contributed by the points in the neighborhood $\vecN_\vecx^{\beta_i}$.

To intuitively illustrate the construction of PILD, we begin with an indicator matrix $\mathbb{I}$ derived from the $\beta$-local depth for a given $\beta$ value. Each entry in the matrix, $\mathbb{I}_{ij}$, is either $1$ or $0$, indicating whether the $j$-th observation in the dataset belongs to the $\beta$-neighborhood of the $i$-th observation. The column sums of $\mathbb{I}$ then represent the frequency with which each point appears in the $\beta$-neighborhoods of each of the point in the dataset. A point with high frequency in the $\mathbb{I}$ matrix can be interpreted as locally central. To refine this measure, the $\mathbb{I}$ matrix can be extended by assigning weights to neighborhood memberships. Specifically, the membership of a point $\vecz$ within the neighborhood of a point $\vecx$ is weighted by the local depth of $\vecx$, as neighborhoods centered on points with higher local depth values should carry greater importance in evaluating local centrality. Furthermore, the PILD matrix, denoted as $\mathbb{P}$, with the $ij$-th entry defined as $PILD(\vecx_j \mid \vecx_i, \vecX)$ for any $\vecx_i$ and $\vecx_j$ in $\vecX$, extends the previous idea of weighted neighborhood membership by further considering the relative order of points within the neighborhood of each observation. This refinement incorporates contributions from points based on their ranking within the neighborhood, providing a more detailed characterization of local relationships. To ensure consistency with $\beta$-ILD and interpretability as PaLD, the PILD matrix normalizes contributions so that each row sum equals the $\beta$-ILD value, i.e., $\sum_{\vecz \in \vecX} PILD(\vecz \mid \vecx, \vecX) = ILD(\vecx \mid \vecX)$.

We summarize some properties of PILD analogue to properties in Definition \ref{def:DATADEPTH} for $\beta$-ILD as follows
\begin{enumerate}[(Q1)]
    \item \textbf{$\mathcal{T}$-invariance}:  If the underlying depth function $D(\vecx\mid P)$ is invariant under any transformation $ T \in \mathcal{T} $, then for any points $ \vecz $ and $ \vecx $ in $ \vecX $,
    $
    PILD(T(\vecz) \mid T(\vecx), T(\vecX)) = PILD(\vecz \mid \vecx, \vecX).
    $
    \item \textbf{Maximality at itself}: For any point $ \vecx $, we have $ PILD(\vecz \mid \vecx, \vecX) \leq PILD(\vecx \mid \vecx, \vecX) $ for all $ \vecz \in \vecX \setminus \{\vecx\} $.
    \item  \textbf{Pseudo-monotonicity}: If the depth function $ D $ is monotone on rays, then PILD satisfies pseudo-monotonicity, which is defined as follows: \textcolor{black}{for any two points $\vecz_1,\vecz_2\in\vecX$, if
$
D\bigl(\vecz_1\mid \vecX_{R\vecx}\bigr)\;\ge\;D\bigl(\vecz_2\mid \vecX_{R\vecx}\bigr),
$
then their contributions to the $\beta$-ILD value of $\vecx$ satisfy
$PILD(\vecz_1\mid \vecx,\vecX)\;\ge\;PILD(\vecz_2\mid \vecx,\vecX).
$ In other words, points that are deeper in $\vecX_{R\vecx}$ contribute more to the $\beta$-ILD of $\vecx$.
}
    \item \textbf{Vanishing outside the $B$-neighborhood}: \textcolor{black}{If the weighting function $ w(\beta) $ is zero for all $ \beta > B $, then for any $\vecz\notin\vecN_\vecx^B$, i.e., the largest local neighborhood, we have that $ PILD^B(\vecz \mid \vecx, \vecX) = 0$.}
\end{enumerate}
Those properties of PILD are straightforward to derive from its definition and the properties of $\beta$-ILD.
\section{Applications}\label{sec:application}
Data depth has been extensively applied to classification tasks, such as maximum depth classifiers and DD-plots-based methods \cite{ghosh2005maximum, li2012dd, lange2014fast}, and to outlier detection techniques leveraging its robust center-outward ranking properties \cite{chen2008outlier, jeong2018outlier, mozharovskyi2022anomaly}. We now demonstrate how the proposed $\beta$-ILD and PILD matrix enhance the performance of data depth in these two key applications.

All experiments in this section utilize the sample version of spatial depth function, due to its computational efficiency and its ability to compute $\beta$-LD values across all locality levels effectively. That is given by
\begin{equation}\label{eq:SD}
SD(\vecz\mid \mathbf{X})=
1-\left\| \sum_{i=1}^n \frac{\vecz-\vecx_i}{n\| \vecz-\vecx_i \|} \right\|,
\end{equation}
where $\vecX=\{\vecx_1,\cdots,\vecx_n\}$.
\subsection{Classification}
The most fundamental depth-based classifier is the max-depth classifier, initially proposed by Liu, Parelius, and Singh in 1999 \cite{liu1999multivariate} and later formally developed by Ghosh and Chaudhuri \cite{ghosh2005maximum}. In essence, given data separated in groups and denoted $\vecX_1, \cdots, \vecX_G$, the max-depth classifier (MD) assigns a new point $\vecz$ to the group for which it achieves the highest depth value, i.e., the assignment is determined by $\argmax_{i=1, \cdots, G} D(\vecz \mid \vecX_i)$. To address more complex classification tasks, advanced classifiers often rely on training separators in the two-dimensional depth-value space known as a DD-plot. The goal is to minimize empirical risk in the training set using a separator, such as a polynomial \cite{li2012dd}, or employing the so-called $\alpha$-procedure, as in the DD$\alpha$ method \cite{lange2014fast}. Beyond directly comparing depth values between different populations, another approach is the depth-based K-nearest neighbor classifier (D-knn) \cite{paindaveine2015nonparametrically}, which is based on the $\beta$-neighborhood from Definition \ref{def:betanbr}. In this approach, a point $\vecz$ is assigned to Group $i$ if its $\vecN_\vecz^\beta$ contains more points from Group $i$ than any other group.

We further propose a new depth-based classification approach based on the PILD matrix, as discussed in Section \ref{sec: PALD}. The PILD matrix is constructed using all points under consideration, regardless of their population or whether their labels are known. For an unlabeled point $\vecz$--the point we want to classify--the proposed PILD-based classifier calculates the average contribution of the point to the local depth values of all points within each population. Specifically, let there be $G$ groups, and denote by $ \vecX_g $ the set of labeled points belonging to the $g$-th group. The classifier assigns the point $ \vecz=\vecx_j $ to the population that maximizes the average PILD contribution, defined as
\begin{eqnarray}
     \hat{m}(\vecx_j) = \argmax_{g \in \{1, 2, \ldots, G\}} \frac{1}{|\vecX_g|}\sum_{i\in\{i: \vecx_i\in\vecX_g\}} \mathbb{P}_{ij},
\end{eqnarray}
where $\mathbb{P}_{ij}=PILD(\vecx_j \mid \vecx_i, \vecX)$ is the $ij$-th entry of the PILD matrix and $|\cdot|$ represents cardinality of a set.

In our study, we evaluate the following depth-based classifiers: the max-depth classifier (MD) and its extensions using $\beta$-LD and $\beta$-ILD (referred to as max LD and max ILD classifiers). For $\beta$-ILD, a uniformly distributed weighting function, as defined in \eqref{eq:weight}, is used. The max ILD classifier includes two variants: a parameter-free version (Full-ILD) when the integration spans all locality levels ($B=1$), and a parameterized version (B-ILD) where the locality parameter $B$ controls the range of integration. We construct the PILD matrix, which similarly leads to two variants of classifiers: a parameter-free version (PILD) when $B=1$ and a parameterized version (B-PILD) with $B$ as the locality parameter. For classifiers based on DD-plots, we apply the DD$\alpha$ method using Full-ILD, $\beta$-LD, B-ILD, PILD, and B-PILD. The DD$\alpha$ method is implemented using the \texttt{ddalpha} R package \cite{pokotylo2016depth} with default configuration. Additionally, we test the D-knn classifier.

For methods requiring parameter selection, $5$-fold cross-validation (CV) is used to determine the optimal values for the locality parameters ($\beta$ for $\beta$-LD and $B$ for B-ILD and B-PILD) and the number of neighbors for D-knn. To provide further comparisons, we also try fixed locality parameters ($0.1$, $0.3$, and $0.5$) for $\beta$-LD and $\beta$-ILD.

\begin{figure*}[t]
     \centering
     \begin{subfigure}[H!]{0.24\textwidth}
         \centering
		\includegraphics[width=\textwidth]{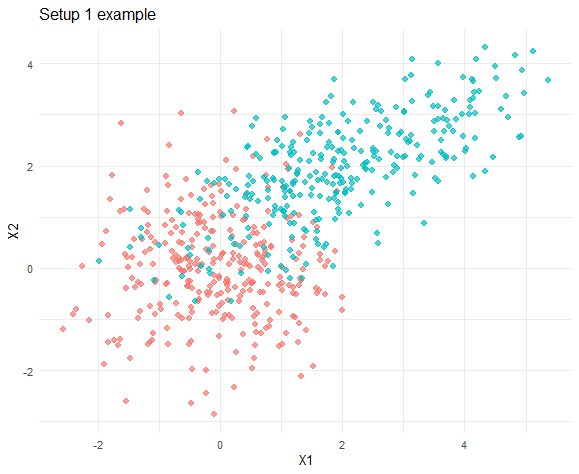}
         \caption{}
     \end{subfigure}
     \hfill
     \begin{subfigure}[H!]{0.24\textwidth}
         \centering
         \includegraphics[width=\textwidth]{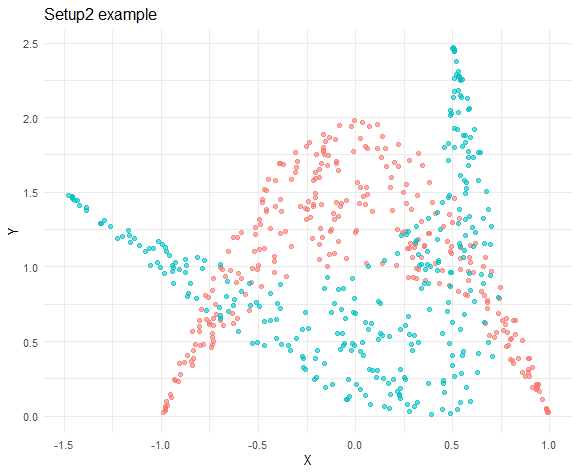}
         \caption{}
     \end{subfigure}
     \hfill
          \begin{subfigure}[H!]{0.24\textwidth}
         \centering
         \includegraphics[width=\textwidth]{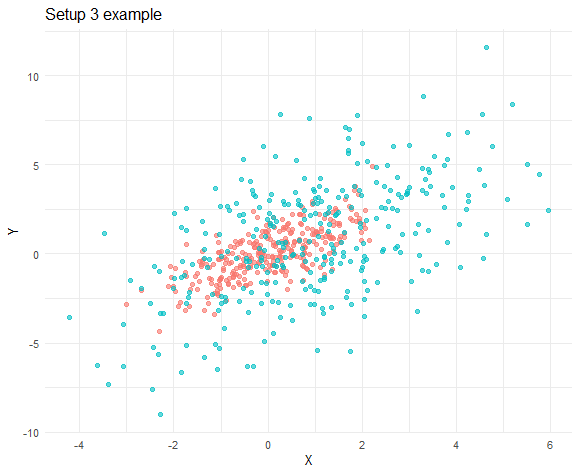}
         \caption{}
     \end{subfigure}
      \hfill
          \begin{subfigure}[H!]{0.24\textwidth}
         \centering
         \includegraphics[width=\textwidth]{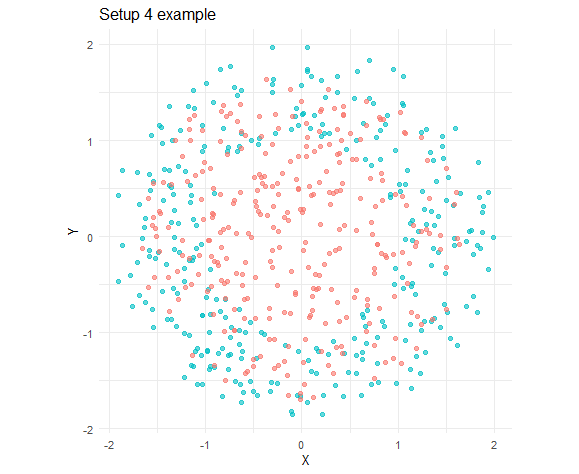}
         \caption{}
     \end{subfigure}
     \caption{Graphs of examples for four setups. $(a)$: Setup $1$; $(b)$: Setup $2$; $(c)$: Setup $3$. $(d)$: Setup $4$.}
     \label{Fig:sgraph}
\end{figure*}
\textcolor{black}{
We consider four different bivariate distributional setups in $\mathbb{R}^2$, with Setup $2$ adopted from Setup $4$ in Paindaveine and Van Bever~\cite{paindaveine2015nonparametrically}.
 For each setup, we simulate a total of $600$ data points drawn from a mixture of two distributions, $f_0$ and $f_1$, with equal prior probabilities $\pi_0 = \pi_1 = 0.5$. That is, each class contributes $300$ points per setup. We randomly select $n_{\text{train}} = 200$ training samples and $n_{\text{test}} = 100$ test samples from each class, yielding $400$ training and $200$ test points in total. This process is repeated $100$ times independently for each setup. Examples of each setup are illustrated in Figure \ref{Fig:sgraph}, with further details summarized below.}
\begin{enumerate}[\textbf{Setup} 1]
    \item \textbf{Bivariate normal distributions.} $ f_i $, $i=0,1$, is the p.d.f. of the normal distribution with mean vector $\mu_i$ and covariance matrix $\Sigma_i$ :
    \[
    \mu_0 = \begin{pmatrix} 0 \\ 0 \end{pmatrix}, \quad \mu_1 = \begin{pmatrix} 2 \\ 2 \end{pmatrix},
    \quad \Sigma_0 = \begin{pmatrix} 1 & 0 \\ 0 & 1 \end{pmatrix}, \quad \Sigma_1 = \begin{pmatrix} 2 & 1 \\ 0 & 1 \end{pmatrix}.
    \]

    \item \textbf{Uniform distributions on half-moons.} $ f_0 $ and $ f_1 $ represent  uniform distributions over the following half-moon shapes:
\[
\vecX_0 = \begin{pmatrix} U \\ V \end{pmatrix} \quad \text{and} \quad \vecX_1 = \begin{pmatrix} -0.5 \\ 2 \end{pmatrix} + \begin{pmatrix} 1 & 0.5 \\ 0.5 & -1 \end{pmatrix} \begin{pmatrix} U \\ V \end{pmatrix},
\]
where $ U \sim \text{Unif}(-1, 1) $ and $ V \mid [U = u] \sim \text{Unif}(1 - u^2, 2(1 - u^2)) $.

    \item \textbf{Overlapping distributions with different point densities.} $ f_i $, $i=0,1$, is the p.d.f. of the overlapping normal distribution :
    \[
    \mu_0 = \begin{pmatrix} 0 \\ 0 \end{pmatrix}, \quad \mu_1 = \begin{pmatrix} 1 \\ 1 \end{pmatrix},
    \quad \Sigma_0 = \begin{pmatrix} 1 & 1 \\ 1 & 2 \end{pmatrix}, \quad \Sigma_1 = \begin{pmatrix} 4 & 4 \\ 4 & 12 \end{pmatrix}.
    \]

    \item \textbf{Ring and ball within the ring.}
\begin{itemize}
    \item $ f_0 $: A uniform distribution over a ring with an inner radius of $ 1 $ and an outer radius of $ 2 $, i.e., $\{\vecx\in\mathbb{R}^{2}:1\leq\|\vecx\|\leq2\}$.
    \item $ f_1 $: A mixture distribution composed of two concentric balls with different radii:
    \begin{itemize}
        \item A large ball with radius $ r_{\text{ball}} = 1.7 $ and prior probability $ 0.7 $.
        \item A small ball with radius $ r_{\text{ball}} = 1 $ and prior probability $ 0.3 $.
    \end{itemize}
    Points are generated in polar coordinates as $ (r \cos(\theta), r \sin(\theta)) $, where $ r \sim \text{Unif}(0, r_{\text{ball}}) $ and $ \theta \sim \text{Unif}(0, 2\pi) $.
\end{itemize}
\end{enumerate}

\begin{figure*}[t]
     \centering
     \begin{subfigure}[H!]{0.47\textwidth}
         \centering
		\includegraphics[width=\textwidth]{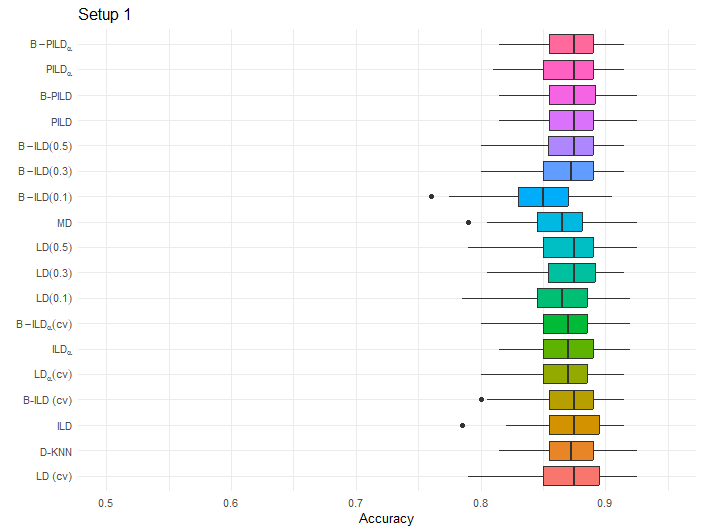}
         \caption{}
     \end{subfigure}
     \hfill
     \begin{subfigure}[H!]{0.47\textwidth}
         \centering
         \includegraphics[width=\textwidth]{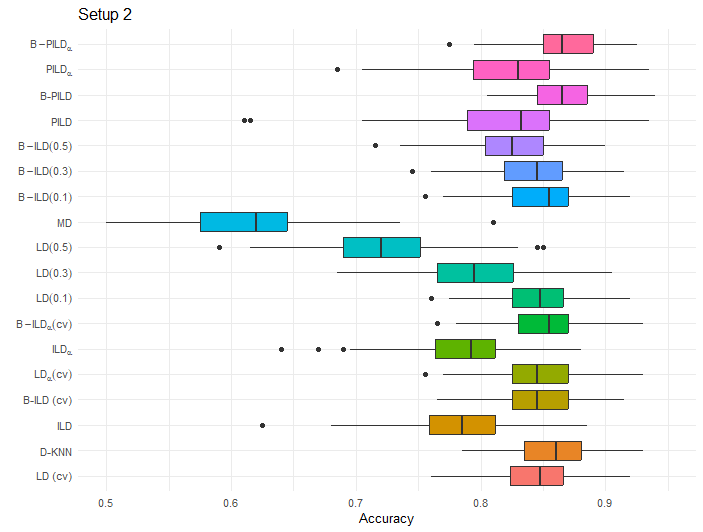}
         \caption{}
     \end{subfigure}
     \hfill
          \begin{subfigure}[H!]{0.47\textwidth}
         \centering
         \includegraphics[width=\textwidth]{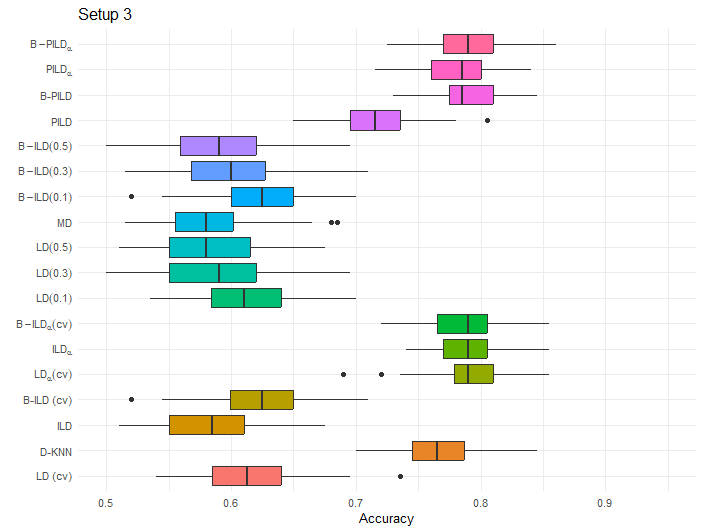}
         \caption{}
     \end{subfigure}
      \hfill
          \begin{subfigure}[H!]{0.47\textwidth}
         \centering
         \includegraphics[width=\textwidth]{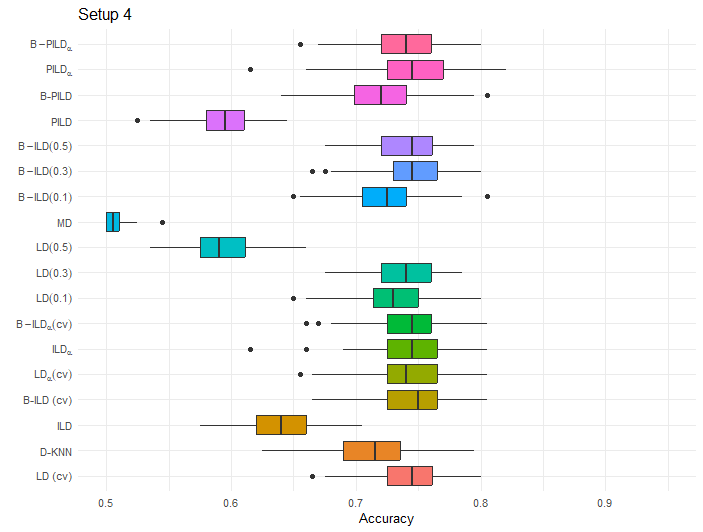}
         \caption{}
     \end{subfigure}
     \caption{Boxplots of classifier accuracies for $4$ setups based on $100$ simulations. Classifiers include: max depth (MD), max LD with $\beta$ set to $0.1, 0.3, 0.5$ and cross-validated (LD($0.1$), LD($0.3$), LD($0.5$), LD(cv)), max Full-ILD (ILD), max ILD with $B$ set to $0.1, 0.3, 0.5$ and cross-validated (B-ILD($0.1$), B-ILD($0.3$), B-ILD($0.5$), B-ILD(cv)), PILD-based classifiers (PILD and B-PILD). DD$\alpha$ classifiers established on some above-mentioned methods (LD$\alpha$(cv), ILD$\alpha$, B-ILD$_\alpha$(cv), PILD$_\alpha$, B-PILD$_\alpha$), and the D-knn with cross-validated $k$ neighbors (D-KNN). The median values of $\beta_{\text{cv}}$, $B^1_{\text{cv}}$ for B-ILD, $B^2_{\text{cv}}$ for B-PILD, and $k_{\text{cv}}$ for each setup are: $\beta_{\text{cv}} = 0.325, 0.1, 0.125, 0.15$, $B^1_{\text{cv}} = 0.5, 0.15, 0.1, 0.2$, $B^2_{\text{cv}} = 0.485, 0.281, 0.389, 0.168$ and $k_{\text{cv}} = 16, 11, 11, 17$. }
     \label{Fig:boxplots}
\end{figure*}

\begin{table*}[ht]
\caption{Mean accuracies of classifiers for each setup and average ranks across all methods.}
\label{Tab:accuracy}
\centering
\begin{adjustbox}{max width=1.1\textwidth}
\begin{tabularx}{1.5\textwidth}{l *{18}{>{\centering\arraybackslash}X}}
\toprule
 Setup &  \shortstack{LD\\ (cv)} & \shortstack{D-KNN} &  \shortstack{ILD} & \shortstack{B-ILD\\ (cv)} &  \shortstack{LD$_\alpha$\\ (cv)} &  \shortstack{ILD$_\alpha$} &  \shortstack{B-ILD$_\alpha$\\ (cv)} &  \shortstack{LD\\ (0.1)} &  \shortstack{LD\\ (0.3)} &  \shortstack{LD\\ (0.5)} &  \shortstack{MD} &  \shortstack{B-ILD\\ (0.1)} &  \shortstack{B-ILD\\ (0.3)} &  \shortstack{B-ILD\\ (0.5)}  &  \shortstack{PILD}  &  \shortstack{B-PILD} &  \shortstack{PILD$_\alpha$} &  \shortstack{B-PILD$_\alpha$}\\
\midrule
Setup 1 & $0.871$ & $0.871$ & $0.872$ & $0.872$ & $0.870$ & $0.871$ & $0.868$ & $0.863$ & $0.871$ & $0.872$ & $0.864$ & $0.849$ & $0.870$ & $0.871$ & $0.873$ & $0.874$ & $0.871$ & $0.871$\\
Setup 2 & $0.846$ & $0.859$ & $0.783$ & $0.846$ & $0.846$ & $0.789$ & $0.851$ & $0.846$ & $0.797$ & $0.720$ & $0.613$ & $0.848$ & $0.841$ & $0.825$ & $0.817$ & $0.866$ & $0.822$ & $0.867$ \\
Setup 3 & $0.613$ & $0.766$ & $0.583$ & $0.624$ & $0.789$ & $0.790$ & $0.787$ & $0.611$ & $0.588$ & $0.582$ & $0.583$ & $0.624$ & $0.599$ & $0.590$ & $0.717$ & $0.789$ & $0.782$ & $0.790$\\
Setup 4 & $0.743$ & $0.712$ & $0.640$ & $0.746$ & $0.742$ & $0.743$ & $0.744$ & $0.730$ & $0.737$ & $0.593$ & $0.505$ & $0.722$ & $0.745$ & $0.741$ & $0.593$ & $0.719$ & $0.745$ & $0.739$\\
\midrule
Rank    &  $7.00$ & $7.50$ & $12.50$ & $4.75$ & $7.25$ & $6.75$ & $7.00$ & $11.50$ & $11.25$ &  $13.50$ & $17.00$ & $11.00$ & $9.50$ & $9.75$ & $9.75$ & $4.75$ & $6.50$ & $4.25$ \\
\bottomrule
\end{tabularx}
\end{adjustbox}
\end{table*}
Based on Figure \ref{Fig:boxplots} and Table \ref{Tab:accuracy}, we observe that all methods perform similarly in Setup $1$, which is characterized by elliptical distributions. In contrast, for setups with non-convex shapes (Setups $2$ and $4$), methods that consider all data points tend to be less effective \textcolor{black}{(e.g., MD, Full-ILD and PILD)}. This finding aligns with the observation that the cross-validated locality parameters $\beta$ and $B$ are small in these setups. Comparing the maximum LD and maximum ILD classifiers with various parameters, we find that the maximum ILD classifiers demonstrate greater robustness, as expected. In cases with large overlapping regions (Setups $3$ and $4$), the $\alpha$-procedure significantly improves classification performance. Notably, while the maximum Full-ILD alone does not perform well in these setups, applying the $\alpha$-procedure boosts its effectiveness, making it highly competitive among all tested classifiers. It is also noteworthy that in Setup $3$, B-PILD is the only method that achieves top-tier performance without the $\alpha$-procedure. Across all setups, among the parameter-free methods, Full-ILD consistently outperforms MD, except in Setup $3$. Similarly, the newly proposed PILD method outperforms MD across all setups. Compared to Full-ILD, PILD demonstrates significantly better performance in Setups $2$ and $3$, though it underperforms in Setup $4$. Overall, the proposed B-PILD and its $\alpha$-variant exhibit the best performance across all setups, highlighting their robustness and adaptability to varying data distributions.

\subsection{Outlier detection}
The objective of this section is not to propose a new outlier detection method but rather to demonstrate how the proposed $\beta$-ILD and its corresponding PILD matrix can improve the performance of  global depth and the $\beta$-LD in outlier detection. In the following experiments, we adopt a straightforward and widely used evaluation approach~\cite{campos2016evaluation}: for each dataset, we assume the proportion of outliers is known. For each method under consideration, we identify outliers as the points with the highest or lowest values of the used criteria (depending on the method) within that specified proportion. Precision is used as the metric to evaluate the performance of different outlier detection methods. Importantly, these experiments are conducted in a unsupervised setting, without splitting the data into training and testing sets.

The depth-based methods considered include $\beta$-LD, $\beta$-ILD, and the column sums of the PILD matrix. For each dataset, these methods share the same locality parameter. When the locality parameter is set to $1$, $\beta$-LD reduces to a global depth measure. To provide context for understanding the performance of depth-based methods, we also include several widely used outlier detection techniques: Local Outlier Factor (LOF) \cite{breunig2000lof}, Isolation Forest (IF) \cite{liu2008isolation}, Natural Outlier Factor (NOF) \cite{huang2016non}, and Relative Density-based Outlier Factor (RDOS) \cite{tang2017local}. LOF, NOF, and RDOS are density-based methods that identify outliers by comparing the density of a point to that of its neighbors.  LOF uses k-nearest neighbors (kNN) to define the reachability distance between points, which is then used to calculate local density for identifying outliers. NOF defines neighborhoods based on the natural neighborhood algorithm, and RDOS refines density estimation by employing kernel density estimation. IF is a binary tree-based method which detects anomalies by isolating points with fewer splits during random partitioning of the feature space. These methods are implemented using the R packages \texttt{dbscan} \cite{dbscan2024} for LOF, \texttt{isotree} \cite{isotree2024} for IF, and \texttt{DDoutlier} \cite{DDout2018} for both NOF and RDOS.

We also explore the potential avenue for incorporating data depth into other algorithms. Because density-based algorithms typically rely on a distance matrix to define neighborhoods, it is natural to consider replacing the Euclidean distance matrix with a depth-based dissimilarity matrix. Here, we introduce two types of similarity matrices based on data depth, with dissimilarities simply defined as $1-\text{Similarity}$.
The first type of similarity matrix has been constructed when defining the $\beta$-neighborhood used for computing $\beta$-LD and $\beta$-ILD. Specifically, the $ij$-th entry represents the depth value of point $\vecx_j$ with respect to the reflection dataset $\vecX_{R\vecx_i}$ (see Definition \ref{def:betanbr}).
The second similarity matrix is constructed from the PILD matrix $\mathbb{P}$. However, the PILD matrix is not directly suitable as a similarity matrix for two key reasons. First, since the row sum of $\mathbb{P}$ equals the $\beta$-ILD value of a point, the scale of the entries varies across rows. Second, the diagonal values of $\mathbb{P}$ may not be the largest in their respective columns, which is inconsistent with a similarity interpretation. To address these issues, we transform the PILD matrix $\mathbb{P}$ into a similarity matrix by normalizing each row by its diagonal element, $\mathbb{P}'_{ij} = \mathbb{P}_{ij}/\mathbb{P}_{ii}$ for all elements, and ensuring symmetry by defining the similarity between points $\vecx_i$ and $\vecx_j$ as $ \min(\mathbb{P}'_{ij}, \mathbb{P}'_{ji})$. We further evaluate the two types of similarity matrices using the LOF method in subsequent experiments, referring to them as LOF$_{\text{SD}}$ and LOF$_{\text{PILD}}$ respectively.

The experiments begin with two toy examples. In each example, the dataset consists of $400$ inliers and $100$ outliers with $100$ runs. The first example consists of a unimodal distribution with outliers scattered around it, while the second example involves a mixture of distributions with outliers positioned between the components. Examples are shown in Figure \ref{fig:toyoutliers}. The detailed settings are as follows:
\begin{itemize}
    \item The inliers are generated from a $t$-Copula using the R package \texttt{copula} \cite{copulas2024}, with Kendall's tau $0.0638$ and degrees of freedom $1$. The outliers are uniformly distributed within the square $[-0.2, 2]$ along both $x$- and $y$-axes, excluding points with a Euclidean distance smaller than $0.1$ from any inlier.
    \item The inliers consist of a mixture of four bivariate normal distributions, each contributing equally to the dataset. The outliers are sampled from another set of bivariate normal distributions positioned between these components. Both inliers and outliers have diagonal covariance matrices:
    \[
    \Sigma_{\text{inliers}} = \begin{bmatrix} 1 & 0 \\ 0 & 1 \end{bmatrix}, \quad 
    \Sigma_{\text{outliers}} = \begin{bmatrix} 2 & 0 \\ 0 & 2 \end{bmatrix}.
    \]
    The inlier means are located at $(4, 4)$, $(-4, 4)$, $(-4, -4)$, and $(4, -4)$. The outlier means are positioned at $(0, 5)$, $(-5, 0)$, $(0, -5)$, and $(5, 0)$.
\end{itemize}
\begin{figure*}[t]
     \centering
     \begin{subfigure}[H!]{0.47\textwidth}
         \centering
		\includegraphics[width=\textwidth]{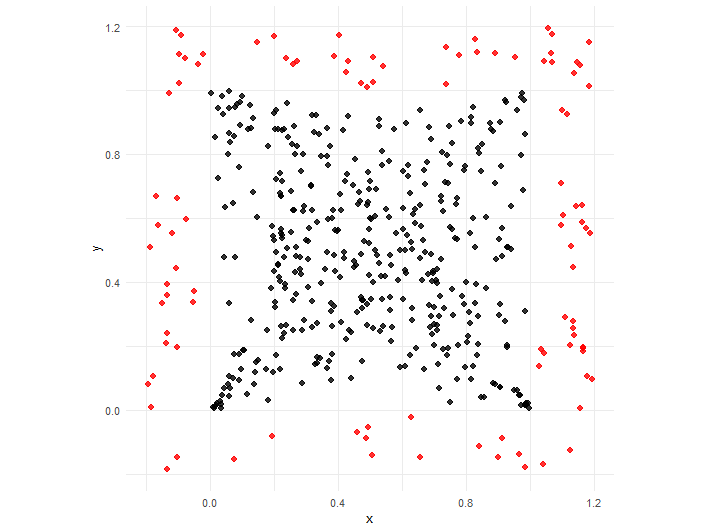}
         \caption{}
     \end{subfigure}
     \hfill
     \begin{subfigure}[H!]{0.47\textwidth}
         \centering
         \includegraphics[width=\textwidth]{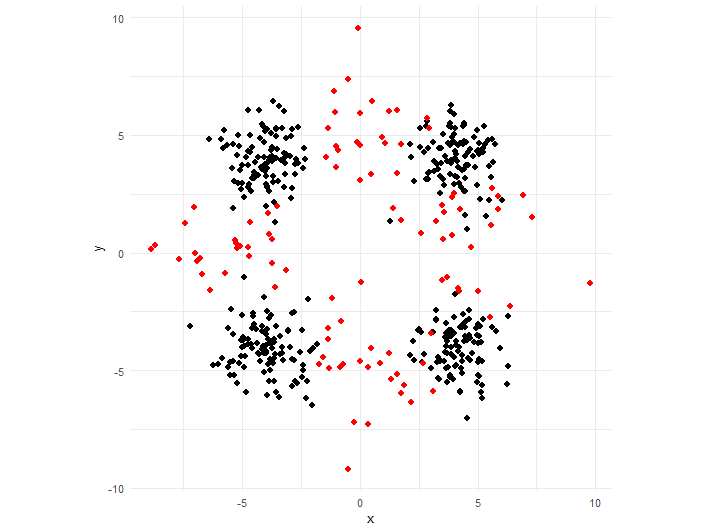}
         \caption{}
     \end{subfigure}
     \caption{Graphs of examples for two toy examples with red points representing outliers. $(a)$: t-Copula; $(b)$ mixtures of normal distributions.}
     \label{fig:toyoutliers}
\end{figure*}
From Figure \ref{fig:outliers}, it is evident that in the two toy examples, $\beta$-ILD consistently outperforms GD and LD, while PILD further improves upon $\beta$-ILD, particularly in the second example. LOF$_\text{PILD}$ performs better than LOF$_\text{SD}$ in the first example, while it is the opposite in the second example. However, their overall differences are minimal, as the simulated data largely conforms to circular shapes, making alternative dissimilarity metrics less impactful.

After the toy examples, we consider five real-world datasets: WPBC, WDBC, Ionosphere, Waveform, and Spambase. These datasets represent a range of scenarios, including small ($n < 500$) and large sample size ($n > 3000$), as well as varying levels of outlier contamination, from low to high. Details of these datasets, including their sources, are provided in the study by Campos et al. \cite{campos2016evaluation} on unsupervised outlier detection. We use the non-duplicate and normalized versions of these datasets. The dataset information and the precision of the methods are summarized in Table \ref{tab:results}. For all methods, we select the optimal parameters that maximize precision. \textcolor{black}{For IF, we fix the number of trees at $500$ and vary the subsampling parameter, which controls the number of data points used per tree, from $4$ to $512$.} For LOF-based methods, all possible values of $ k $ are tested up to $\min\left(n/2, 2000\right)$. For RDOS, the Gaussian kernel bandwidth is fixed at $2$, and $ k $ is tested up to $200$. For depth-based methods, we evaluate locality levels of $0.1$, $0.3$, $0.5$, and $1$.

\begin{figure*}[ht!]
     \centering
     \begin{subfigure}[H!]{0.47\textwidth}
         \centering
		\includegraphics[width=\textwidth]{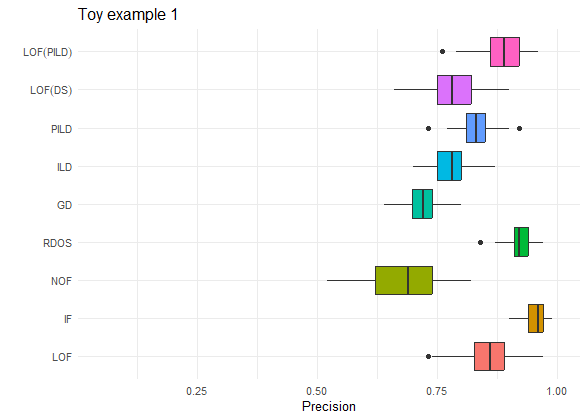}
         \caption{}
     \end{subfigure}
     \hfill
     \begin{subfigure}[H!]{0.47\textwidth}
         \centering
         \includegraphics[width=\textwidth]{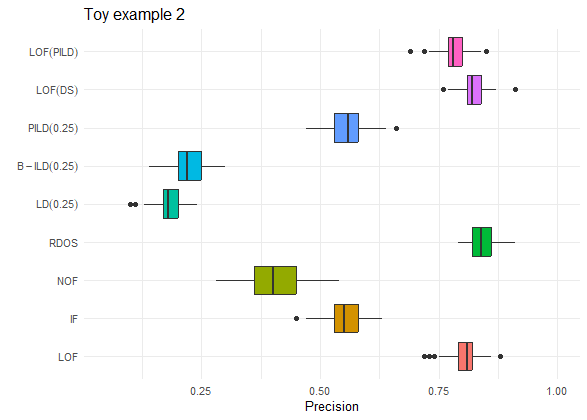}
         \caption{}
     \end{subfigure}
     \caption{Boxplots of precision in anomaly detections for two toy examples based on $100$ simulations. $(a)$: the $k$'s in LOF, RDOS, LOF$_\text{SD}$ and LOF$_\text{PILD}$ are $20$. $(b)$: the $k$'s in LOF, RDOS, LOF$_\text{SD}$ and LOF$_\text{PILD}$ are $91$, $105$, $61$ and $71$ respectively. The locality level for all methods is $0.25$.}
     \label{fig:outliers}
\end{figure*}

From Table \ref{tab:results}, there are no obvious differences in performance between $\beta$-LD and $\beta$-ILD when considering the real world data sets. Nevertheless, PILD consistently outperforms both. LOF$_\text{SD}$ surpasses LOF for WPBC, Waveform, and SpamBase, while LOF$_\text{PILD}$ excels for WPBC, WDBC, and SpamBase. Notably, LOF$_\text{PILD}$ achieves the best results for WPBC and WDBC, while LOF$_\text{SD}$ stands out as the top-performing method for Waveform, showing clear improvements over the original LOF. Interestingly, for the Ionosphere dataset with $B=0.1$, LOF$_\text{PILD}$ underperforms compared to the column sums of the PILD matrix. This result may be due to the simplified structure of the PILD-based similarity matrix: points within each other's $\beta_1$-neighborhoods have a similarity of $1$, while points outside a given point's $B$-neighborhood have a similarity of $0$. With a small $B$, this simplification can overly reduce the complexity of relationships between points.  
\begin{table}[h!]
\centering
\caption{\textcolor{black}{Precision comparison on benchmark datasets: GD, LD, ILD represent global depth, $\beta$-LD and $\beta$-ILD. For each dataset, all depth-based methods use the same locality level. Specifically, the locality level is set to $1$ for WPBC and WDBC, and to $0.1$, $0.3$, and $0.5$ for Ionosphere, Waveform, and SpamBase, respectively.}}
\label{tab:results}
\begin{adjustbox}{max width=\textwidth}
\begin{tabular}{lccc|ccccccccc}
\toprule
\textbf{Dataset} & \textbf{n} & \textbf{d} & \textbf{\% Outliers} & \textbf{LOF} & \textbf{IF} & \textbf{NOF} & \textbf{RDOS} & \textbf{GD/LD} & \textbf{ILD} & \textbf{PILD} & \textbf{LOF$_\text{SD}$} & \textbf{LOF$_\text{PILD}$} \\
\midrule
WPBC        & $198$   & $33$  & $23.74\%$   & $0.2128$   & $0.1702$   & $0.1702$   & $0.2340$   & $0.1489$   & $0.1702$   & $0.2128$   & $0.2340$   & $\mathbf{0.3404}$   \\
WDBC        & $367$   & $30$  & $2.72\%$   & $0.6$   & $\mathbf{0.7}$   & $0.2$   & $0.6$   & $0.6$   & $0.5$   & $0.6$   & $0.6$   & $\mathbf{0.7}$   \\
Ionosphere  & $351$   & $33$  & $35.90\%$   & $\mathbf{0.8254}$   & $0.6667$   & $0.3492$   & $0.7619$   & $0.8016$   & $0.8016$   & $0.8175$   & $0.8016$   & $0.8095$   \\
Waveform    & $3443$   & $21$  & $2.90\%$   & $0.22$   & $0.08$   & $0$   & $0.21$   & $0.06$   & $0.06$   & $0.11$   & $\mathbf{0.32}$   & $0.16$   \\
SpamBase    & $4601$   & $57$  & $39.40\%$   & $0.4131$   & $\mathbf{0.6387}$   & $0.3491$   & $0.4479$   & $0.3900$   & $0.3900$   & $0.4181$   & $0.4153$   & $0.4263$   \\
\bottomrule
\end{tabular}
\end{adjustbox}
\end{table}

\section{Conclusion} \label{sec:conclu}
In this paper, we introduce a novel local depth definition, $\beta$-integrated local depth, built upon $\beta$-local depth. $\beta$-ILD generalizes $\beta$-LD, retaining compatibility with most depth definitions while offering greater robustness. We analyze its theoretical properties inherited from global depth and $\beta$-LD, and establish its partitioned local depth form by constructing the PILD matrix, which draws inspiration from local community depth. Furthermore, we highlight the interpretability of the column sums of the PILD matrix, demonstrating their effectiveness as a measure of local centrality.

We consider the approachability of these new concepts in classification and outlier detection, showing significant improvements over traditional global and local depth methods. However, $\beta$-ILD incurs higher computational costs. While spatial depth allows efficient computation of $\beta$-ILD by supporting incremental updates as $\beta$ grows, other depth definitions, such as projection depth, require full recomputation with varing $\beta$ as the local neighborhood size changes, making them much more computationally intensive.

Future research should focus on developing sampling-based or approximation techniques to mitigate these computational challenges, thereby enabling the broader application of $\beta$-ILD and PILD in large-scale and high-dimensional scenarios.
\section*{Appendix Proofs}
\label{app1}
This appendix contains the proofs of properties and propositions stated in this paper. Note that many properties of $\beta$-ILD are deduced from those of $\beta$-LD via the Dominated Convergence Theorem (DCT), provided that the corresponding property holds for $\beta$-LD. Here, we provide detailed proofs for properties of $\beta$-LD not covered in detail in \cite{paindaveine2013depth}, and omit proofs that follow directly by applying the DCT to existing results (e.g., Property~\ref{prop:consis}, which follows from Theorem~\ref{thm:consis1}).
\begin{proof}[Proof of Proposition \ref{pro:1}]
The maximum jump between successive local depth values is represented as
\begin{equation}\label{eq:LDupdf}
\Delta = \max_{1 < i \le b} |LD(\beta_i) - LD(\beta_{i-1})|.
\end{equation}
For the sample $\beta$-ILD, we analyze the difference between $ILD(\beta_b)$ and $ILD(\beta_{b-1})$ for a point $\vecx$
\begin{eqnarray}\notag
    ILD(\beta_b) - ILD(\beta_{b-1}) &=& \sum_{i=1}^{b} LD(\beta_i) w_B(\beta_i) - \sum_{i=1}^{b-1} LD(\beta_i) w_{B'}(\beta_i), \\\notag
    &=& \frac{\sum_{i=1}^{b} LD(\beta_i)}{b} - \frac{\sum_{i=1}^{b-1} LD(\beta_i)}{b-1}, \\
    \notag
    &=& \frac{LD(B)}{b} - \frac{\sum_{i=1}^{b-1} LD(\beta_i)}{b(b-1)}, \\
    &=& \frac{\sum_{i=1}^{b-1} \left( LD(B) - LD(\beta_i) \right)}{b(b-1)}.
\label{eq:ILD-ILD}
\end{eqnarray}
For each term of above sum, note that 
$
|LD(\beta_b) - LD(\beta_i)| \le (b - i) \Delta.
$ Therefore, we obtain the bound
\[
|ILD(\beta_b) - ILD(\beta_{b-1})| \le \frac{1}{b(b-1)} \sum_{i=1}^{b-1} (b - i) \Delta
= \frac{1}{b(b-1)} \cdot \frac{b(b-1)}{2} \Delta
= \frac{1}{2} \Delta.
\]
The same reasoning applies to any pair $(\beta_i, \beta_{i-1})$ with $1 < i \le b$, as among them, the difference in $\beta$-LD values will not exceed \eqref{eq:LDupdf}, hence we can conclude that
\[
\max_{1 < i \le b} |ILD(\beta_i) - ILD(\beta_{i-1})| \le \frac{1}{2} \Delta.
\]
For the second inequality, from \eqref{eq:ILD-ILD} again,
$$
\bigl|ILD(\beta_b) - ILD(\beta_{b-1})\bigr|\le\frac{(b-1)\Delta^\dagger}{b(b-1)}=\frac{\Delta^\dagger}{b}.
$$
\end{proof}
Several of the following proofs rely on the next lemma, which characterizes the behavior of symmetric mixture distributions under affine transformations.
\begin{lemma}
\label{lem:affine_mixture}
Let $\vecX \sim P_{\vecX}$ be a random vector in $\mathbb{R}^d$, and fix $\vecx \in \mathbb{R}^d$. Define the reflected random variable $\vecX' := 2\vecx - \vecX$, and let $P_{\vecX'}$ be its distribution. Consider the symmetric mixture distribution
\[
Q_{\vecX} := 0.5\, P_{\vecX} + 0.5\, P_{\vecX'}.
\]
Let $T(\vecz) = \mathbb{A} \vecz + \vecc$ be any non-singular affine transformation, where $ \mathbb{A} $ is a non-singular matrix representing a linear transformation, and $ \vecc $ is a translation vector.. Then,
\[
Q_{T(\vecX)} = 0.5\, P_{T(\vecX)} + 0.5\, P_{T(\vecX')}.
\]
\end{lemma}
\begin{proof}
Let $f_{P_{\vecX}}$ and $f_{P_{\vecX'}}$ denote the densities of $\vecX \sim P_{\vecX}$ and $\vecX' \sim P_{\vecX'}$, respectively, where $f_{P_{\vecX'}}(\vecz) = f_{P_{\vecX}}(2\vecx - \vecz)$. The mixture distribution $Q_{\vecX}$ has density
\begin{equation}\label{eq:qdense}
f_{Q_{\vecX}}(\vecz) = 0.5\, f_{P_{\vecX}}(\vecz) + 0.5\, f_{P_{\vecX}}(2\vecx - \vecz).
\end{equation}
For the distribution $Q_{T(\vecX)}$, the density is
\[
f_{Q_{T(\vecX)}}(\vecy) = \frac{1}{|\det \mathbb{A}|}\, f_{Q_{\vecX}}(\mathbb{A}^{-1}(\vecy - \vecc)).
\]
Substituting \eqref{eq:qdense}, we obtain
\[
f_{Q_{T(\vecX)}}(\vecy) = \frac{1}{|\det \mathbb{A}|} \left[ 0.5\, f_{P_{\vecX}}(\mathbb{A}^{-1}(\vecy - \vecc)) + 0.5\, f_{P_{\vecX}}(2\vecx - \mathbb{A}^{-1}(\vecy - \vecc)) \right].
\]
Now consider transforming $ P_\vecX $ and $ P_{\vecX'} $ separately under $ T $. We have
\[
f_{P_{T(\vecX)}}(\vecy) = \frac{1}{|\det \mathbb{A}|} f_{P_{\vecX}}(\mathbb{A}^{-1}(\vecy - \vecc)), \quad
f_{P_{T(\vecX')}}(\vecy) = \frac{1}{|\det \mathbb{A}|} f_{P_{\vecX}}(2\vecx - \mathbb{A}^{-1}(\vecy - \vecc)).
\]
Then their mixture yields
\[
0.5\, f_{P_{T(\vecX)}}(\vecy) + 0.5\, f_{P_{T(\vecX')}}(\vecy) = f_{Q_{T(\vecX)}}(\vecy).
\]
Hence,
$
Q_{T(\vecX)} = 0.5\, P_{T(\vecX)} + 0.5\, P_{T(\vecX')}.
$
\end{proof}
\begin{proof}[Proof of Property \ref{prop:ctn}]
        Let $(\vecx_n)_{n \geq 1}$ be a sequence in $\mathbb{R}^d$ such that $\vecx_n \to \vecx$ as $n \to \infty$, and let $\beta \in (0,1]$ be fixed. Since the integrated local depth (ILD) is defined as an integral of the local depth (LD), we should show
    \[
    LD^\beta(\vecx_n \mid P) \xrightarrow[n \to \infty]{} LD^\beta(\vecx \mid P),
    \]
    which is equivalent to:
    \[
    D(\vecx_n \mid P^{\beta}_{\vecx_n}) \xrightarrow[n \to \infty]{} D(\vecx \mid P^\beta_\vecx).
    \]
     The central region $R^\beta_{\vecx}$ corresponding to the symmetrized distribution $r(P_{\vecX}, \vecx) = \frac{1}{2}P_{\vecX} + \frac{1}{2}P_{2\vecx - \vecX}$ is formally represented as
\[
R^\beta_{\vecx} = \bigcap_{\alpha \in A(\beta)} R_\alpha(r(P_\vecX, \vecx)), \quad \text{where} \quad A(\beta) = \{\alpha : P(R_\alpha(r(P_\vecX, \vecx))) \geq \beta\}.
\]
Here, $R_\alpha(r(P_\vecX, \vecx))$ represents the set of points $\vecx \in \mathbb{R}^d$ for which the depth $D(\vecx \mid r(P_\vecX, \vecx)) \geq \alpha$.

Since $P$ is absolutely continuous, it is enough to show $P^\beta_{\vecx_n}\xrightarrow[n \to \infty]{} P^\beta_\vecx$, where each distribution is obtained by conditioning $P$ on the regions $R^\beta_{\vecx_n}$ and $R^\beta_{\vecx}$, respectively. We establish this result in two steps. In the first step, we establish that
\begin{eqnarray}
    D(\vecx_n \mid \frac{1}{2}P_{\vecX} + \frac{1}{2}P_{2\vecx_n - \vecX}) \xrightarrow[n \to \infty]{} D(\vecx \mid \frac{1}{2}P_{\vecX} + \frac{1}{2}P_{2\vecx - \vecX}).
\end{eqnarray}
In the second step, we consider the restrictions to $R^\beta_{\vecx_n}$ and $R^\beta_{\vecx}$.\\
\begin{enumerate}[Step 1]
    \item Now, consider the translation $T(\vecX) = \vecX + (\vecx - \vecx_n)$. Since $D$ is invariant under similarity transformations and by Lemma~\ref{lem:affine_mixture}, we can equivalently write
\[
D(\vecx_n \mid \frac{1}{2}P_{\vecX} + \frac{1}{2}P_{2\vecx_n - \vecX}) = D(\vecx \mid \frac{1}{2}P_{\vecX - \vecx_n + \vecx} + \frac{1}{2}P_{\vecx_n + \vecx - \vecX}),
\]
As $n \to \infty$, both $P_{\vecX - \vecx_n + \vecx}$ and $P_{\vecx_n + \vecx - \vecX}$ converge weakly to $P_{\vecX}$ and $P_{2\vecx - \vecX}$, respectively, so the mixture converges accordingly
\[
\frac{1}{2}P_{\vecX - \vecx_n + \vecx} + \frac{1}{2}P_{\vecx_n + \vecx - \vecX} \xrightarrow[n \to \infty]{} \frac{1}{2}P_{\vecX} + \frac{1}{2}P_{2\vecx - \vecX}.
\]
This implies
\[
D(\vecx_n \mid \frac{1}{2}P_{\vecX} + \frac{1}{2}P_{2\vecx_n - \vecX}) = D\left(\vecx \mid \frac{1}{2}P_{\vecX - \vecx_n + \vecx} + \frac{1}{2}P_{\vecx_n + \vecx - \vecX}\right) \xrightarrow[n \to \infty]{} D(\vecx \mid \frac{1}{2}P_{\vecX} + \frac{1}{2}P_{2\vecx - \vecX}).
\]
\item According to the relationship between depth function convergence and central regions, as summarized by Dyckerhoff \cite{dyckerhoff2016convergence}, if $D$ is strictly monotonous and its central region is convex, the central regions will converge in the Hausdorff sense (While the original proof addresses sample distributions converging to the theoretical distribution as $n \to \infty$, it relies on constructing convergent sequences and a proof by contradiction, which is also applicable to our setting).

Next, consider a Borel set $\mathcal{B} \subset \mathbb{R}^d$, there is $P(\mathcal{B} \mid R^\beta_\vecx) = \frac{P(\mathcal{B} \cap R^\beta_\vecx)}{P(R^\beta_\vecx)}$. The Hausdorff convergence $R^\beta_{\vecx_n} \xrightarrow[n \to \infty]{H} R^\beta_{\vecx}$ implies that for any $\vecz \in R^\beta_{\vecx}$, there exists $N_0 \in \mathbb{N}$ such that for all $n \geq N_0$, $\vecz \in R^\beta_{\vecx_n}$. Consequently, we have
\[
\mathbf{1}_{R^\beta_{\vecx_n}}(\vecz) \to \mathbf{1}_{R^\beta_\vecx}(\vecz), \quad \text{and} \quad \mathbf{1}_{\mathcal{B} \cap R^\beta_{\vecx_n}}(\vecz) \to \mathbf{1}_{\mathcal{B} \cap R^\beta_{\vecx}}(\vecz),
\]
as $n\to\infty$, where $\mathbf{1}$ represents the indicator function.
Since $P$ is absolutely continuous, its associated probability density function $f$ is bounded and integrable, and we can apply the DCT,
\[
\int_{\mathcal{B} \cap R^\beta_{\vecx_n}} f(\vecz) \, d\vecz \to \int_{\mathcal{B} \cap R^\beta_{\vecx}} f(\vecz) \, d\vecz \quad \text{and} \quad \int_{R^\beta_{\vecx_n}} f(\vecz) \, d\vecz \to \int_{R^\beta_{\vecx}} f(\vecz) \, d\vecz.
\]
Thus, $P^\beta_{\vecx_n} \xrightarrow[n \to \infty]{} P^\beta_{\vecx}$ holds if $\vecx_n$ is in the support of $P$, i.e.,$P(R^\beta_{\vecx_n}) > 0$ for any $\beta > 0$.
\end{enumerate}
Finally, since the $\beta$-LD is continuous with respect to $\beta \in (0,1]$ and bounded, we apply the DCT to conclude that the $\beta$-integrated local depth $ILD(\vecx \mid P)$ is also continuous at $\vecx$.
\end{proof}
\begin{proof}
\textbf{Property \ref{prop:inva}}:
Assume the depth function $D(\vecx \mid P)$ is invariant under any $T \in \mathcal{T}$. Let $Q_{\vecX}$, $P_{\vecX}$ and $P_{\vecX'}$ be defined as in Lemma~1. By Lemma~1, for any non-singular affine transformation $T$, we have $Q_{T(\vecX)} = 0.5\, P_{T(\vecX)} + 0.5\, P_{T(\vecX')}$, and hence
\begin{equation}\label{eq:affine}
D(\vecx \mid Q_{\vecX}) = D(T(\vecx) \mid 0.5\, P_{T(\vecX)} + 0.5\, P_{T(\vecX')}).
\end{equation}
Since the central region of a depth function inherits the invariance properties of the depth function \cite{zuo2000structural}, it follows from \eqref{eq:affine} that for $\beta$-LD,
\begin{equation}
LD(\beta \mid T(\vecx), P_{T(\vecX)}) = D(T(\vecx)\mid P^\beta_{T(\vecx)}) = D(\vecx\mid P^\beta_{\vecx}) = LD(\beta \mid \vecx, P_{\vecX}).
\end{equation}
By the previous equality and $W$ being invariant under $T$, $\beta$-ILD inherits the affine invariance property.
\end{proof}

\begin{proof}\textbf{Property \ref{prop:central}}:
\[
ILD^B(\vecx \mid P) = \int_0^B LD^\beta(\vecx \mid P) \, dW(\beta).
\]

By the continuity of the local depth function $LD^\beta(\vecx \mid P)$ with respect to $\beta$, we know that $LD^\beta(\vecx \mid P)$ is bounded and continuous on $(0, B]$. The weighting measure $W$ satisfies $W((0, B]) = 1$. Then, the Mean Value Theorem for weighted integrals ensures that there exists some $\beta^*\in(0,B]$ such that
$$
\frac{\int_0^B LD^\beta(\vecx\mid P)\,w(\beta)\,d\beta}
     {\int_0^B w(\beta)\,d\beta}
\;=\;
LD^{\beta^*}(\vecx\mid P).
$$
Since $W((0,B])=\int_0^B w(\beta)\,d\beta=1$, this immediately gives
$$
\int_0^B LD^\beta(\vecx\mid P)\,dW(\beta)
=LD^{\beta^*}(\vecx\mid P).
$$
As $B\to0$, the integration interval $(0,B]$ shrinks and hence $\beta^*\to0$. Combining this with the ``centrality in extreme locality level'' result for $\beta$-LD from \cite{paindaveine2013depth}, the statement is proved.

\end{proof}
\begin{proof}\textbf{Proposition \ref{pro:2}}
Assume that for any distribution $P$, the depth function satisfies
\[
\lim_{\|\vecx\| \to \infty} D(\vecx \mid P) = 0.
\]
For the $\beta$-LD, observe that $P^\beta_{\vecx}$ is simply $P$ restricted to the $\beta$-central region $R^\beta_{\vecx}(P)$ and then renormalized. 
\begin{itemize}
    \item \textbf{Probability-based depth}: For any depth function $D$ based on probability content, the following argument holds.
\begin{equation}\label{eq:ineqp4}
D(\vecx\mid P^\beta_{\vecx}) \le \frac{1}{P(R^\beta_{\vecx}(P))} D(\vecx\mid P) = \frac{1}{\beta} D(\vecx\mid P).
\end{equation}
To prove this argument, we take halfspace depth as an example, and others can be shown similarly.
\[
D(\vecx\mid P) = \inf\{P(H) : H \text{ is a closed halfspace with } \vecx \in \partial H\}.
\]
For a fixed $\beta \in (0,1]$, the halfspaces are now constrained within $R^\beta_{\vecx}(P)$, so
\begin{align*}
    LD^\beta(\vecx\mid P) = D(\vecx\mid P^\beta_{\vecx}) 
    &= \inf\left\{\frac{P(H \cap R^\beta_{\vecx}(P))}{P(R^\beta_{\vecx}(P))} : H \text{ is a closed halfspace with } \vecx \in \partial H\right\} \\
    &\le \frac{1}{\beta} \inf\{P(H) : H \text{ is a closed halfspace with } \vecx \in \partial H\} \\
    &= \frac{1}{\beta} D(\vecx\mid P).
\end{align*}
By \eqref{eq:ineqp4}, for any probability-based depth, it follows that
\[
\lim_{\|\vecx\| \to \infty} LD^\beta(\vecx \mid P) = \lim_{\|\vecx\| \to \infty} D(\vecx \mid P^\beta_{\vecx}) \leq \frac{1}{\beta}\lim_{\|\vecx\| \to \infty} D(\vecx\mid P) = 0.
\]
\item \textbf{Type B depth}: Type B depth is in the form
$
D(\vecx\mid P) = \frac{1}{1 + \mathbb{E}_{\vecy \sim P}[h(\vecy, \vecx)]},
$
where $h$ is a non-negative, unbounded, and center-outward monotone increasing function (e.g., $h(\vecx, \vecy)=\|\vecy-\vecx\|_p$ for $L^p$ depth).

By Jensen’s inequality,
\[
D(\vecx\mid P)\leq \mathbb{E}_{\vecy \sim P}\left[\frac{1}{1 + h(\vecy, \vecx)}\right]=\int \frac{1}{1 + h(\vecy, \vecx)}\, dP(\vecy).
\]
Similarly, for the local version,
\[
D(\vecx \mid P^\beta_\vecx) \leq \frac{1}{\beta}\int_{R^\beta_\vecx} \frac{1}{1 + h(\vecy, \vecx)}\, dP(\vecy) \leq \frac{1}{\beta}\int \frac{1}{1 + h(\vecy, \vecx)}\, dP(\vecy).
\]

Since $h(\vecy, \vecx)$ is unbounded and center-outward monotone increasing,
$
\frac{1}{1 + h(\vecy, \vecx)} \xrightarrow[\|\vecx\| \to \infty]{\text{a.s.}} 0.
$
By the DCT,
\[
\lim_{\|\vecx\| \to \infty} \int \frac{1}{1 + h(\vecy, \vecx)}\, dP(\vecy) = \int \lim_{\|\vecx\| \to \infty} \frac{1}{1 + h(\vecy, \vecx)}\, dP(\vecy) = 0.
\]
Therefore, $
\lim_{\|\vecx\| \to \infty} D(\vecx \mid P^\beta_\vecx) = 0.
$

\item \textbf{Spatial depth}: Recall the spatial depth function $
D(\vecx \mid P) = 1 - \left\| \mathbb{E}_{\vecy \sim P} \left[ \frac{\vecy - \vecx}{\|\vecy - \vecx\|} \right] \right\|.
$ The vanishing-at-infinity property of it implies $\lim_{\|\vecx\| \to \infty}\left\|\mathbb{E}_{\vecy \sim P}\left[\frac{\vecy-\vecx}{\|\vecy-\vecx\|}\right]\right\| \to 1$. For the local version we have $
LD^\beta(\vecx \mid P) = 1 - \left\| \mathbb{E}_{\vecy\sim P^\beta_\vecx}  \left[ \frac{\vecy - \vecx}{\|\vecy - \vecx\|} \right] \right\|.
$ Now let $g(\vecy, \vecx) = \frac{\vecy - \vecx}{\|\vecy - \vecx\|}$. Then,
\[
\mathbb{E}_{\vecy \sim P}[g(\vecy, \vecx)] = \int_{R^\beta_\vecx} g(\vecy, \vecx) dP(\vecy) + \int_{(R^\beta_\vecx)^c} g(\vecy, \vecx) dP(y),
\] where $(R^\beta_\vecx)^c$ denotes the complement of $R^\beta_\vecx$. Note
$
\mathbb{E}_{\vecy\sim P^\beta_\vecx}[g(\vecy, \vecx)] = \frac{1}{\beta} \int_{R^\beta_\vecx} g(\vecy, \vecx) dP(\vecy),
$
so that
\[
\mathbb{E}_{\vecy \sim P}[g(\vecy, \vecx)] = \beta \mathbb{E}_{\vecy\sim P^\beta_\vecx}[g(\vecy, \vecx)] + \int_{(R^\beta_x)^c} g(\vecy, \vecx) dP(y).
\]
Applying the triangle inequality and the Jensen's inequality,
\begin{eqnarray}\notag
    \left\|\mathbb{E}_{\vecy \sim P}[g(\vecy, \vecx)] \right\| &\leq& \beta \left\| \mathbb{E}_{\vecy\sim P^\beta_\vecx}[g(\vecy, \vecx)] \right\| + (1-\beta),\\\notag
     \left\| \mathbb{E}_{\vecy\sim P^\beta_\vecx}[g(\vecy, \vecx)] \right\| &\geq& \frac{1}{\beta} \left\| \mathbb{E}_{\vecy\sim P}[g(\vecy, \vecx)] \right\| -\frac{1}{\beta}+1.
\end{eqnarray}

As $\lim_{\|\vecx\| \to \infty} \left\| \mathbb{E}_{\vecy \sim P}[g(\vecy, \vecx)] \right\| = 1$, it follows that
$
\lim_{\|x\|\to\infty} \left\| \mathbb{E}_{\vecy\sim P^\beta_\vecx}[g(\vecy, \vecx)] \right\| \geq 1.
$
Since $\left\| \mathbb{E}_{\vecy \sim P^\beta_{\vecx}}[g(\vecy, \vecx)] \right\|$ is bounded above by $1$, its limit must be $1$. Thus,
\[
\lim_{\|\vecx\| \to \infty} LD^\beta(\vecx \mid P) = 1- \lim_{\|\vecx\| \to \infty}\left\| \mathbb{E}_{\vecy\sim P^\beta_\vecx}[g(\vecy, \vecx)] \right\| = 0.
\]
\end{itemize}

\end{proof}

 \bibliographystyle{elsarticle-num} 
\bibliography{references}





\end{document}